\documentclass[11pt,twoside]{article}

\usepackage{amsmath}
\usepackage{amssymb}
\usepackage{mathrsfs}
\usepackage{amsthm}

\usepackage{latexsym}

\usepackage{indentfirst}
\usepackage{color}
\usepackage{txfonts}

\usepackage{anysize}

\allowdisplaybreaks

\usepackage[colorlinks=true,
  linkcolor=red,
  citecolor=blue,
  urlcolor=magenta]{hyperref}

\def\R{{\mathbb R}}
\def\rn{{{\R}^n}}
\def\D{\mathcal{D}}
\def\d{\mathrm{d}}
\def\M{\mathcal{M}}
\def\H{\mathcal{H}}

\def\sgn{\operatorname{sgn}}

\newtheorem{theorem}{Theorem}[section]
\newtheorem{lemma}[theorem]{Lemma}

\newtheorem{proposition}[theorem]{Proposition}
\newtheorem{example}[theorem]{Example}
\theoremstyle{definition}
\newtheorem{remark}[theorem]{Remark}
\newtheorem{definition}[theorem]{Definition}

\numberwithin{equation}{section}

\begin{document}
\title{\bf\Large Preduals of weighted homogeneous Bourgain-Morrey Besov-Triebel-Lizorkin  spaces associated with  operators
\footnotetext{\hspace{-0.35cm} 2020 {\it
Mathematics Subject Classification}. Primary 42B35; Secondary  42B25, 46E30. \endgraf
{\it Key words and phrases}. Besov space, Triebel-Lizorkin space, Bourgain-Morrey space,  Space of homogeneous type, predual
\endgraf
The work is supported by the National Natural Science Foundation of China (Grant No.
12561002), Hainan Provincial Natural Science Foundation of China (Grant No. 126MS0135) and the Science and Technology Project of Guangxi (Guike AD25069086).
}}
\date{}
\author{}
\author{Tengfei Bai, Pengfei Guo and Jingshi Xu\footnote{Corresponding author,
E-mail: \texttt{jingshixu@126.com}}}
\maketitle

\vspace{-0.8cm}

\begin{center}
\begin{minipage}{13cm}
{\small {\bf Abstract:}\quad
	Let $(X,\rho,\mu)$ be a space of homogeneous type satisfying $\mu(X) =\infty$, the doubling property and the reverse doubling condition.	 Let $L$ be a nonnegative self-adjoint operator on $L^2(X)$	whose heat kernel has  a Gaussian upper bound.
	First, we study the preduals of weighted Bourgain-Morrey spaces on $X$ and their properties, such as lattice property and completeness.
	Then we introduce the weighted homogeneous Besov-Triebel-Lizorkin block spaces.  It is shown that  they are the preduals of  weighted homogeneous Bourgain-Morrey Besov-Triebel-Lizorkin  spaces associated with the operator $L$.  Consequently, these block spaces are independent of  the choice of  partitions of unity.
}
\end{minipage}
\end{center}

\vspace{0.2cm}


\section{Introduction}
Let  $ (X, \rho, \mu) $ be a  space of homogeneous type, with quasidistance $\rho$ and  nonnegative Borel measure $\mu$ on $X$, which has the doubling property below. 
For $x\in X$ and $r>0$, denote by $B(x,r) = \{ y\in X : \rho(x,y) < r\}$  the open ball with radius $r>0$ and center $x \in X$. Set $V(x,r) = \mu ( B (x,r )) $.  We say that $\mu $ has the {\it doubling property} if there exists a constant $C>0$ such that for all $x \in X$ and $r>0$, 
\begin{equation} \label{double}
	V(x,2r) \le C V(x,r).
\end{equation}
The doubling property (\ref{double}) yields a constant $n>0 $ (may not be an integer) playing the role of a dimension such that 
\begin{equation*} 
	V(x, \lambda r) \le C \lambda ^n  V(x,  r), 
\end{equation*}
for all $\lambda \ge 1$, $x\in X$ and $r>0$, and that 
\begin{equation*}
	V(x,r) \le C \Big( 1+ \frac{\rho(x,y)}{r} \Big)^{\tilde n}  V(y,r),
\end{equation*}
for all $x,y \in X$, $r>0$ and for some $\tilde{n} \in [0,n]$; see for example \cite{CW71}.
For the function spaces of this paper, we  also assume that $X$  satisfies the {\it reverse doubling condition}; see Remark \ref{dya cube}. 

Let $L$ be a nonnegative self-adjoint operator on $L^2 (X)$ which generates a semigroup $\{ e^{-\alpha L} \} _{ \alpha>0 }$. 
Let  $D(L)$ be the domain  of $L$. By the fact that the generator of a strongly continuous semigroup is a closed 	and densely defined linear operator that determines the semigroup uniquely (\cite[page 51]{EN20}),   $D(L)$ is dense in  $L^2 (X)$.
Let $p_\alpha (x,y)$ be the kernel of the semigroup $ e^{-\alpha L} $. In this paper, we suppose that the kernel  $p_\alpha (x,y)$  satisfies a Gaussian upper bound, that is, there exist constants $C, c >0$ such that for all $x,y \in X$ and $\alpha >0$, 
\begin{equation*}
	p_\alpha (x,y) \le \frac{C}{\mu( B(x, \sqrt{\alpha} )  )}  \exp \bigg(   -\frac{\rho(x,y) ^2 }{c \alpha}   \bigg) . 
\end{equation*}

Smoothness spaces on $\rn$ have been researched extensively in the last 50 years in various areas of mathematics such as in the study of PDE's, in Harmonic Analysis and in Approximation Theory; we refer the reader  to  monographs \cite{Tri83,Tri92,Tri06,Tri20}  for the systematic study of Besov-Triebel-Lizorkin spaces.
Replacing the Laplace operator $- \Delta$ by a nonnegative self-adjoint operator on $L^2 (X)$, the theory of  Besov and Triebel-Lizorkin spaces associated with differential operators $L$ has been developed by many mathematicians in various settings; see \cite{BBD20, BDY12, BD15, BD17, CKP12,  GHS23, KP15, KPPX09, LYY16, PX08, YY14, SYY24}. 

Indeed, the characterizations of the Besov-Triebel-Lizorkin spaces on $\rn$   induced by Hermite expansions  in terms of the needlet coefficients were obtained in  \cite{PX08}.
 Besov and Triebel-Lizorkin spaces  associated with a non-negative self-adjoint operator $L$ whose heat kernel enjoys Gaussian localization  in a spaces of homogeneous type were introduced in \cite{BDY12, CKP12, KP15}.
The Musielak-Orlicz-Hardy space  associated with  a nonnegative self-adjoint operator satisfying the Davies-Gaffney estimates on  a metric space with doubling measure was introduced in \cite{YY14}.
The homogeneous and inhomogeneous  Besov and Triebel-Lizorkin type spaces associated to Hermite operators were introduced by Bui and Duong in \cite{BD15}.
The inhomogeneous Besov-type and Triebel-Lizorkin spaces associated to a nonnegative self adjoint operator which satisfies sub-Gaussian upper bound estimate, H\"older continuity, and stochastic completeness were considered in \cite{LYY16}. 
The homogeneous  Besov and Triebel-Lizorkin  spaces  in terms of the Gauss-Weierstrass semi-group generated by the  Laguerre operator were introduced by Bui and Duong in \cite{BD17}.
In \cite{BBD20}, Bui et al. established  the theory of weighted Besov-Triebel-Lizorkin spaces   associated with the operator, whose kernel   enjoys only  Gaussian upper bound.

A special case of Bourgain-Morrey spaces was first  introduced by Bourgain in \cite{B91}. 
After then, 
Bourgain-Morrey spaces have been applied to PDE's, especially the Strichartz estimate and nonlinear Schr\"odinger equations; see, e.g., \cite{BPV07,MV98,MVV99}.

The preduals  of  Bourgain-Morrey type spaces have been researched by many authors.
In \cite{M16}, Masaki introduced the block spaces as the  preduals of Bourgain-Morrey spaces.
Hatano et al. \cite{HNSH23} investigated Bourgain--Morrey function spaces on $\rn$ from the viewpoints of harmonic analysis and functional analysis. They established the duality of Bourgain-Morrey space.
In \cite{ZSTYY23}, Zhao et al.  introduced Besov-Bourgain-Morrey spaces and obtained their preduals.
Triebel-Lizorkin-Bourgain-Morrey spaces and their preduals were studied by Hu, Li, and Yan in \cite{HLY23}.
 Zhang et al. obtained the  generalized grand Besov-Bourgain-Morrey spaces and their preduals in \cite{ZYZ24}.
 Very recently, in \cite{ZYY26}, Zhu, Yang and Yuan introduced Besov-Bourgain-Morrey spaces on the space of homogeneous type satisfying the reverse doubling property. They also established the predual and the dual spaces of these spaces.

In \cite{BX25}, the first author and the third author of the paper introduced the weighted homogeneous Bourgain-Morrey Besov spaces and Triebel-Lizorkin spaces associated with the operator. 
Recently, we introduced the weighted homogeneous Bourgain-Morrey Besov type spaces and Triebel-Lizorkin type spaces associated with the operator in  \cite{BGX26}.   Their continuous   characterizations	in terms of  Peetre maximal functions,   noncompactly supported functional calculus, and heat kernels were obtained.

Motivated by the above literature, we  establish the preduals of the weighted homogeneous Bourgain-Morrey Besov spaces and Triebel-Lizorkin spaces associated with operators.
The paper is organized as follows.
In Section \ref{prel}, we recall notations such as  dyadic cubes, Muckenhoupt weights, Hardy-Littlewood maximal functions, the class of distributions and   function spaces; and some lemmas.
In Section \ref{predual BM}, we study the predual of a weighted Bourgain-Morrey space on $X$.
The preduals of weighted homogeneous Bourgain-Morrey Besov spaces associated with the operators are obtained in Section \ref{Predual BM Beosv}.
In Section \ref{Preudal BM TL}, we establish the preduals of weighted homogeneous Bourgain-Morrey Triebel-Lizorkin spaces associated with the operators.

Throughout this paper, we let $c, C$ denote constants that are independent of the main parameters involved but whose value may differ from line to line. 
Let $\mathbb N = \{0,1,2,3,\ldots \}$ and $\mathbb N_+ = \{1,2,3,\ldots \}$.
Let $\mathbb Z$ be the set of all integers.
Let $\chi_{E}$ be the characteristic function of the set $E\subset X$.
Given a ball $B:=B(x_B, r_B)$ and $\lambda>0$, $\lambda B$ denotes the ball with the same center as $B$ whose radius is $\lambda$ times that of $B$. 
For $A,B \ge 0$, by $A\lesssim B$ we mean that $A\leq CB$ with some positive constant $C$ independent of appropriate quantities. 
By $ A \approx B$, we mean that $A\lesssim B$ and $B\lesssim A$.
Let $ \mathscr S (\mathbb R)$ be the class of Schwartz functions on $\mathbb R$.
For a quasi-Banach space $Y$,  its dual $Y^\ast$ is defined as the space of all continuous linear functionals $T$ on $Y$ equipped with the norm $ \|T\|_{ Y^\ast} = \sup_{\| f\|_Y \le 1  }   | T (f) |$. Define the function $\sgn$ by $\sgn (f) = f/|f| $ if $|f|\neq 0$ and $\sgn (f) = 0$ if $f =0$.

\section{Preliminaries}\label{prel}

\subsection{Dyadic cubes}
The following covering lemma comes from \cite{C90}.
\begin{lemma}\label{cube}
	Let $\mu (X)  =\infty$.
	There exists a collection of open sets $\{  Q_\tau ^k \subset X : k \in \mathbb Z, \tau \in I_k  \}$, where $I_k$ denotes certain  index set depending on $k$ and constants $\gamma \in (0,1)$, $a_0 \in (0,1]$ and $ \kappa_0 \in (0,\infty)$ such that
	\begin{itemize}
		\item[\rm (i)] $\mu (  X \backslash \cup_{ \tau \in I_k  } Q_\tau^k) =0 $ for all $k \in \mathbb Z$;
		\item[\rm (ii)] for every $\tau, \beta$, if $i\ge k$, then either $Q_\tau^i \subset Q_\beta^k$ or $Q_\tau^i \cap Q_\beta^k  = \emptyset $;
		\item[\rm (iii)] for every $(k,\tau)$ and each $i < k$, there exists a unique $\tau '$ such that $Q_\tau^k \subset Q_{\tau '}^i$; we say that $ Q_{\tau '}^i  $ is a parent of cube $Q_\tau^k$.
		\item[\rm (iv)] the diameter of  $Q_\tau ^k$ is less than or equal to $\kappa _0 \gamma ^k$;
		\item[\rm (v)] each $Q_\tau ^k $ contains certain ball $B(x_{Q_\tau ^k}, a_0 \gamma ^k )$.
	\end{itemize}
\end{lemma}
\begin{remark}\label{dya cube}
	Since the constants $\gamma$ and $a_0$ are not essential in the paper, without loss of generality, we may assume that $\gamma = a_0 = 1/2$. Then fix a collection of open sets in Lemma \ref{cube} and denote this collection by $\mathcal D$. We call open sets in $\mathcal D$ the dyadic cubes in $X$ and $  x_{Q_\tau ^k}$ the center of the cube $Q_\tau ^k \in \mathcal{D}$. Denote $\mathcal D _\nu : = \{  Q_\tau ^{\nu +1} \in \mathcal D : \tau \in I_{\nu+1}  \}$ for each $\nu \in \mathbb Z$. Then for each cube $Q \in \mathcal D_\nu$, we have $ B(x_Q, c_0 2^{-\nu} )  \subset Q \subset   B(x_Q, \kappa_0 2^{-\nu} ) =: B_Q $, where $c_0 $ is a constant independent of $Q$.
	
	Note that the diameter of cubes in $\mathcal D_k$ approximates $2^{-k}$. Hence the parent cube $Q^{k - \nu}_{\tau'}$ has  at most $c2^{\nu n}$ cubes in $\mathcal D _k$, where $c$ is a not important positive number.
	
	For each cube $Q \in \mathcal D$, let $ Q ^{1th}$ be its parent cube. Since $X$ satisfies the doubling condition and the relation $B(x_Q, c_0 2^{-\nu} )  \subset Q \subset   B(x_Q, \kappa_0 2^{-\nu} )$, we have
	\begin{equation*}
		\mu(Q ^{1th} ) \le C 2^n \mu (Q),
	\end{equation*}
	where $n$ plays the role of a dimension of $X$. 
\end{remark}
In the sequel, we further suppose that $X$ satisfies    the {\it  reverse doubling condition}. That is,  for a dyadic cube $Q $ and its parent $R$, there exists a constant $\alpha_0 \in (0,1) $ such that 
\begin{equation*}
	\mu (Q) \le \alpha_0 \mu (R).
\end{equation*} 
And  $\alpha_0$ is called the reverse doubling  constant of $X$. From \cite[Proposition 2.2]{CKP12}, if $X$ is connected, then the reverse doubling condition holds.
Hence in this article,  $X$  is a space of homogeneous type, with quasidistance $\rho$ and  nonnegative Borel measure $\mu$ on $X$, which satisfies the doubling property,  the reverse doubling condition, and $\mu (X) = \infty$.

\subsection{Muckenhoupt  weights}
A weight function $\omega$ is a locally integrable function on $X$ that takes values in $(0,\infty)$ almost everywhere. For a given weight function $\omega$ and a measurable set $E \subset X$, we denote  the weighted measure of $E$ by $\omega(E)$, where $\omega(E)=\int_E \omega(x)\,\mathrm {d} \mu (x)$ and $V(E) = \mu(E)$.
For $1\le p \le \infty $, $p' $ is the conjugate exponent of $p$, that is $1/p + 1/p' =1$.

\begin{definition}
	Let $1<p<\infty $. We say that a weight $\omega$ belongs to class $A_p$ if
	\begin{align*}
		[\omega]_{A_p} & := \sup_{B\;  \mathrm{ balls \; in  } \;X} \Big(   \frac{1}{V(B)} \int_B \omega(x) \mathrm {d} \mu(x)    \Big)  \Big(   \frac{1}{V(B)} \int_B \omega(x) ^ { -1/(p-1) } \mathrm {d} \mu (x)   \Big) ^{ (p-1 ) }   <\infty.
	\end{align*}
	We call a weight $\omega$   an $A_1$ weight if
	\begin{equation*}
		\mathcal M (w) (x) \le C w(x)
	\end{equation*}
	for almost all $x \in X$ and some $C>0$, where  $\mathcal M$ is the usual Hardy-Littlewood maximal function. That is 
	\begin{equation*}
		\mathcal M (f) (x) = \sup_{x \ni B } \frac{1}{V(B) } \int_B |f(x) | \d \mu (x) .
	\end{equation*}
	Set $A_\infty = \cup_{p\ge 1} A_p$.
\end{definition}

For  $\omega\in A_\infty $ and $ 0<p<\infty  $, define the weighted Lebesgue space $L^p(\omega)$ by
\begin{equation*}
	\bigg\{   f : \| f\| _{  L^p(\omega)}  = \bigg(  \int_X |f(x)|^p \omega (x)\mathrm {d} \mu(x)     \bigg) ^{1/p} <\infty      \bigg\}.
\end{equation*}

In the following lemma, 
we list some properties  of $A_p$ weights, whose proofs are  similar to those in \cite[Chapter 7]{Gra1}.

\begin{lemma} \label{weights} 
	Let $\omega \in A_p$ for some $1\le p <\infty$. Then
	\begin{itemize}
		\item[\rm (i)] The classes $A_p$ are increasing as $p$ increases.
		\item[\rm (ii)] The measure $\omega (x) \mathrm {d} \mu (x) $ is doubling: for all cubes $Q$ and all $\mu$-measurable subsets $A$ of $Q$ we have
		\begin{equation*}
			\Big(    \frac{\mu(A)}{ \mu(Q) } \Big) ^p \le [ \omega]_{A_p}  \frac{\omega (A)  }{  \omega(Q) }   .
		\end{equation*}
		\item[\rm (iii)] Let $0<\alpha <1$. Then there exists $\beta : = 1-   {(1-\alpha)^p} / {[\omega]_{A_p} }  <1$ such that  whenever $S$ is a measurable subset of a cube $Q$ with $  \mu (S) \le \alpha \mu (Q)$, we have $ \omega (S) \le \beta \omega (Q).$
		\item[\rm (iv)] If $ \omega \in  A_p, 1 < p <\infty,$ then there exists $1<r < p<\infty$ such that $\omega \in A _r$.
		\item[\rm (v)]A weight $\omega \in A_p$ if and only if both $\omega$ and $\omega ^{-1/ (p-1)} $ are in $A_\infty$.
	\end{itemize}
\end{lemma}
For $\omega \in A_\infty$,  we define $q_\omega  = \inf \{ q: \omega \in A_q   \}$. The self-improvement	property of the Muckenhoupt weights shows that $\omega \in A_{q_\omega}$ can never happen unless $\omega \in A_1$.

\subsection{Functional calculus and the class of distributions}
Fix $x_0 \in X$ as a reference point  in $X$. The class of test functions $\mathcal S$ associated with $L$ is defined as the set of all functions $\phi \in \cap _{m\ge 1} D(L^m) $ such that
\begin{equation*}
	\mathcal P _{m, \ell } (\phi) = \sup_{x\in X}  (1+ \rho(x, x_0) )^m  | L^\ell \phi (x) | <\infty, \; \forall m >0, \ell \in \mathbb{N}.
\end{equation*}
From \cite{KP15}, we know that $\mathcal S$ is a complete locally convex space with topology generated by the family of seminorms  $ \{  \mathcal P _m^\ell : m >0, \ell \in \mathbb N  \}. $
The space of distributions $\mathcal S ' $   is   the set of all continuous linear functional on $\mathcal S$ with the dual defined by
\begin{equation*}
	\langle f , \phi \rangle = f ( \bar \phi)
\end{equation*}
for all $f\in \mathcal S' $ and $\phi \in \mathcal S$.

Following \cite{GKKP17} , we define the space $\mathcal S _\infty$ as the set of all functions $\phi \in \mathcal S$ such that for each $k \in \mathbb N$, there exists $g_k \in \mathcal S$ such that $\phi = L^k g_k $. 
That is, $L^{-k}\phi \in \mathcal S $  for all $k\ge 1$.
Note that such an $g_k$, if exists, is unique; see \cite{GKKP17}. The topology in $\mathcal S _\infty $  is generated by the following family of seminorms:
\begin{equation*}
	\mathcal P _{m,\ell, k} ^{*} (\phi) = \mathcal P_{m,\ell} (g_k), \; m >0; \ell , k \in  \mathbb N
\end{equation*}
where $\phi = L^k g_k$. Denote by $\mathcal S _\infty ' $ the set of all continuous linear functionals on $\mathcal  S_\infty $.
We also define
\begin{equation*}
	\mathcal P_m = \{  g \in \mathcal S ' : L^m g = 0 \}, \; m\in \mathbb N
\end{equation*}
and set $\mathcal P = \cup _{m\in \mathbb N} \mathcal P _m$. The class $\mathcal P$ is defined as the class of generalized polynomials
associated with $L$. From \cite[Proposition 3.7]{GKKP17}, we have the identification $\mathcal S ' / \mathcal P = \mathcal S _\infty '$.
Note that if $L = - \Delta$, the Laplacian on $\rn$, the distributions in $\mathcal S ' / \mathcal P = \mathcal S _\infty ^{'} $ are identical with the classical tempered distributions modulo polynomial.

Denote by $E_L(\lambda)$ a spectral decomposition of the operator $L$. By spectral theory, for any bounded Borel function $F : [0,\infty) \to \mathbb C$, we can define
\begin{equation*}
	F(L) = \int_0^\infty F(\lambda) \d E_L(\lambda)
\end{equation*}
as a bounded operator on $L^2 (X)$. 
From \cite[Section 2.2]{GKKP17},
if $L$  is non-negative self-adjoint operator that maps real-valued to real-valued functions, then for any real-valued, measurable and bounded function $F$ on $\mathbb R_+$  the operator $F (L)$ is bounded on $L^2 (X)$, self-adjoint, and maps real-valued functions to real-valued functions. 
If $F(L)$ is an integral operator, then its
kernel $K_{F( L) }  ( x,y)$ is real-valued and $K_{F( L) }  ( x,y) = K_{F( L) }  ( y,x) . $

From \cite[Lemma 2.6]{BBD20}, we can see that if $\varphi \in \mathscr S (\mathbb R)$ (the class of Schwartz functions on $\mathbb R$) with supp $\varphi \subset(0,\infty)$, then $K_{  \varphi ( t \sqrt {L} ) }  (x, \cdot) \in \mathcal S _\infty $ and  $K_{  \varphi ( t \sqrt {L} ) }  (\cdot, y) \in \mathcal S _\infty $. Hence, for $f\in \mathcal S _\infty ' $ we can define
\begin{equation} \label{varphi}
	\varphi( t \sqrt{L}  ) f(x) = \langle f, K_{  \varphi( t \sqrt{L} ) }  (x, \cdot) \rangle .
\end{equation}
Note that (\ref{varphi}) was introduced first in \cite{GKKP17}.

\subsection{Function spaces}
In \cite{BX25}, first author and third author of this paper introduce the weighted Bourgain-Morrey space on $X$ and the weighted homogeneous Bourgain-Morrey Besov-Triebel-Lizorkin  space associated with the operator $L$ where $L$ is  a nonnegative self-adjoint operator on $L^2 (X)$ satisfying a Gaussian upper bound on its heat kernel.

\begin{definition}
	Let $\mathcal D$ be the dyadic cubes in $X$ as in Remark \ref{dya cube}.
	Let $0<p\le t<\infty$
	, $0<r\le\infty$ and $\omega \in A_\infty$. Define $M_{p,\omega}^{t,r}  :=  M_{p,\omega}^{t,r}  (X)$ as the space of $f\in L_{\mathrm{loc}}^{p}(\omega)$
	such that
	\[
	\|f\|_{M_{p,\omega}^{t,r}} :=\bigg\|  \bigg\{\omega (Q_\tau ^k  ) ^{1/t-1/p}\bigg(\int_{Q_\tau ^k } |f(x)|^{p}\omega(x) \mathrm {d} \mu(x)  \bigg)^{1/p}\bigg\}_{k \in\mathbb{Z}, \tau \in I_k}\bigg\|_{\ell^{r}}<\infty.
	\]
\end{definition}

\begin{definition} \label{def spa 1}
	Let $\mathcal D$ be the dyadic cubes in $X$ as in Remark \ref{dya cube}.
	Let $\psi$ be a partition of unity,
	 that is,  $\psi \in \mathscr (\mathbb R)$ with supp $\psi \subset [1/2,2]$, $\int_\mathbb R \psi (\xi) \xi^{-1 } \d \xi \neq 0 $,  and
	 \begin{equation*}
	 	\sum_{ j \in \mathbb Z}  \psi_j ( \lambda )  = 1 \quad {\rm on} \quad (0,\infty) 
	 \end{equation*}
 where $ \psi_j : =  \psi (2^{-j} \lambda ) $ for each $j \in \mathbb Z$. 	
	Let $ 0< q\le \infty$, $s\in \mathbb R$ and $\omega \in A_\infty $. Let $0<p<t<r<\infty$ 	or $0<p\le t<r=\infty$.
	The weighted homogeneous Bourgain-Morrey Besov  space $\dot{B}^{s,q,\psi,L}_{p,t,r,\omega} (X)$ is defined as follows:
	\[
	\dot{B}^{s,q,\psi,L}_{p,t,r,\omega} (X) = \{ f\in \mathcal S_\infty ' : \| f \| _{ \dot{B}^{s,q,\psi,L}_{p,t,r,\omega} (X)} <\infty  \},
	\]
	where
	\[
	\| f\|_{\dot{B}^{s,q,\psi,L}_{p,t,r,\omega} (X)} =
	\bigg(  \sum_{j\in \mathbb{Z}}  2^{js}\Big\|  \psi_j \big(\sqrt{L}\big)f\Big\|^q_{M_{p,\omega}^{t,r}}   \bigg)^{1/q}.
	\]
	The weighted homogeneous Bourgain-Morrey	 Triebel-Lizorkin  space $\dot{F}^{s,q,\psi,L}_{p,t,r,\omega} (X)$ is defined by
	\[
	\dot{F}^{s,q,\psi,L}_{p,t,r,\omega} (X) = \{ f\in \mathcal S_\infty ' : \| f \| _{ \dot{F}^{s,q,\psi,L}_{p,t,r,\omega} (X)} <\infty  \},
	\]
	where
	\[
	\| f\|_{\dot{F}^{s,q,\psi,L}_{p,t,r,\omega} (X)} =  \bigg\|\bigg( \sum_{j\in \mathbb{Z}} 2^{jsq} \Big| \psi_j \big(\sqrt{L}\big)f \Big|^q\bigg)^{1/q}\bigg\|_{M_{p,\omega}^{t,r}}   .
	\]
\end{definition}

Next, we consider the nontriviality of $M_{p,\omega}^{t,r} $.

\begin{lemma} \label{nontrival}
	Let $ \omega\in A_\infty $. 	Let $\beta_\omega : =  1 - (1- \alpha_0 ) ^{ q_\omega +A } / [\omega ]_{A_{  q_\omega +A }} $ where $A >0$.
	Let $ 0< p< t < \min\{ r, \log_2 ( \beta_\omega ^{-1} )  r  /n \}  <\infty$ or let $ 0 <p \le t <r =\infty $.
	Then for all $Q \in \D$,
	$\|  \chi_{ Q^{k_0}_{\tau_0} } \|_{M_{p,\omega}^{t,r}} <\infty  $, and for all $x\in X$ and $ R \in (0,\infty) $, $ \chi_{ B(x,R) } \in M_{p,\omega}^{t,r}$.

\end{lemma}

\begin{proof}
	Since $\omega \in A_\infty$, for any $A >0 $, $\omega \in A_{ q_\omega +A }$.
	Let $\beta_\omega : =  1 - (1- \alpha_0 ) ^{ q_\omega +A } / [\omega ]_{A_{  q_\omega +A }}  \in (0,1)  $.
	
	Case $ r<\infty $.   we write out the norm in full:
	\begin{equation*}
			\|  \chi_{ Q^{k_0}_{\tau_0} } \|_{M_{p,\omega}^{t,r}} ^r  =\left(  \sum_{ k \ge k_0 } +\sum_{ k < k_0 }  \right)  \sum_{\tau \in I_k} \omega (Q_\tau ^k  ) ^{r/t-r/p} \bigg(\int_{Q_\tau ^k } \chi_{ Q^{k_0}_{\tau_0} } \omega(x) \mathrm {d} \mu(x)  \bigg)^{r/p} 
	\end{equation*}
	As in Remark \ref{dya cube}, there are at most $c2^{(  k-k_0 )  n}$ dyadic cubes $ Q ^{k }_\tau    $ such that $Q ^{k}_{\tau}  \subset Q ^{k_0}_{\tau_0}  $ where $c$ is a constant depending only on $X$.
	Since $ \omega \in A_{ q_\omega +A } $,  we have
	\begin{equation*}
		\omega ( Q ^{k }_\tau  )  \lesssim \beta_\omega  ^{ k- k_0 } \omega ( Q ^{k_0}_{\tau_0}  ) .
	\end{equation*} 
for every $ Q ^{k }_\tau  \subset  Q ^{k_0}_{\tau_0}  $.
If $ n + ( \log_2 \beta_\omega )r/t <0 $,
\begin{align*}
	& \sum_{ k \ge k_0 }  \sum_{\tau \in I_k} \omega (Q_\tau ^k  ) ^{r/t-r/p} \bigg(\int_{Q_\tau ^k } \chi_{ Q^{k_0}_{\tau_0} } \omega(x) \mathrm {d} \mu(x)  \bigg)^{r/p} \\
	& = \sum_{ k \ge k_0 }    \sum_{\tau \in I_k , Q_\tau ^k  \subset Q^{k_0}_{\tau_0}} \omega (Q_\tau ^k  ) ^{r/t}  \\
	& \lesssim \sum_{ k \ge k_0 } 2^{(  k-k_0 )  n}  \beta_\omega  ^{ ( k- k_0 ) r/t } \omega ( Q ^{k_0}_{\tau_0}  ) ^{r/t} \\
	& \lesssim  \omega ( Q ^{k_0}_{\tau_0}  ) ^{r/t} .
\end{align*}
If $k <k_0$, there is only one $ \tau_k  $ such that $Q ^{k_0}_{\tau_0}  \subset Q ^k _{\tau_k} $.
	Since $ \omega \in A_{ q_\omega +A } $,  we have
\begin{equation*}
\omega ( Q ^{k_0}_{\tau_0}  ) 	  \lesssim \beta_\omega  ^{ k_0 - k } \omega ( Q ^{k }_\tau  ).
\end{equation*} 
for every $ Q ^{k_0}_{\tau_0}   \subset Q ^{k }_\tau   $. Since $1/t -1/p <0 $, we have
\begin{equation*}
	\omega ( Q ^{k }_\tau  ) ^{1/t -1/p }  \lesssim   \beta_\omega  ^{ (k_0 - k)  (1/t-1/p ) }  \omega ( Q ^{k_0}_{\tau_0}  ) ^{1/t -1/p}.
\end{equation*}
Since $ p<t $, we obtain
\begin{align*}
		& \sum_{ k <  k_0 }   \sum_{\tau \in I_k} \omega (Q_\tau ^k  ) ^{r/t-r/p} \bigg(\int_{Q_\tau ^k } \chi_{ Q^{k_0}_{\tau_0} } \omega(x) \mathrm {d} \mu(x)  \bigg)^{r/p} \\
	& = \sum_{ k < k_0 }  \omega (Q_{\tau_k} ^k  ) ^{r/t-r/p}   \omega ( Q^{k_0}_{\tau_0} ) ^{r/p}  \\
	& \lesssim \sum_{ k < k_0 }   \beta_\omega  ^{ (k_0 - k)  (r/t-r/p ) }  \omega ( Q ^{k_0}_{\tau_0}  ) ^{r/t -r/p}   \omega ( Q^{k_0}_{\tau_0} ) ^{r/p} \\
	& \lesssim  \omega ( Q^{k_0}_{\tau_0} ) ^{r/t} .
\end{align*}
Hence we have
\begin{equation*}
	\|  \chi_{ Q^{k_0}_{\tau_0} } \|_{M_{p,\omega}^{t,r}} \lesssim \omega ( Q^{k_0}_{\tau_0} ) ^{1/t} .
\end{equation*}
The proof of case $r =\infty$  is similar. 

For any given ball $B$, there exist a constant $N(B) \in \mathbb N_{+} $ and a sequence $ \{ Q _i\} _{ i=1}^{N (B) } $ of dyadic cubes such that $ B \subset \bigcup_{ i=1} ^{N (B) } $ and hence $ \chi_B \in  M_{p,\omega}^{t,r}$. This completes the proof of  Lemma \ref{nontrival}.
\end{proof}

\begin{remark}
	Case $\omega \equiv 1$ can be founded in  \cite[Theorem 2.16]{ZYY26}.  If $X =\rn$ and $\omega \equiv 1$, then $\beta _\omega = 2^{-n}$. Hence Lemma \ref{nontrival} coincides with \cite[Example 2.9]{HNSH23}.
\end{remark}

\subsection{Lemmas}

\begin{lemma} [Lemma 3.2, \cite{BDL18}]  \label{basic est}
	Let $\epsilon >0$.
	
	{\rm (i)} For any $1\le p \le \infty $, we have
	\begin{equation*}
		\bigg(   \int_X     \Big(  1+ \frac{\rho(x,y)}{\alpha}  \Big)^{-(n+ \epsilon)  p}  \mathrm {d} \mu(y)   \bigg)^{1/p} \lesssim V(x, \alpha)^{1/p}
	\end{equation*}
	for all $x\in X$ and $\alpha > 0$.
	
	{\rm (ii)} For all $f\in L_{\rm {loc}}^1 (X)$, we have
	\begin{equation*}
		\int_X \frac{1}{  V(x   \wedge y, \alpha) }  \Big(  1+ \frac{\rho(x,y)}{\alpha}  \Big)^{-n-\epsilon} |f(y)| \mathrm {d} \mu (y) \lesssim \mathcal M f(x)
	\end{equation*}
	for all $x\in X$ and $\alpha > 0$.
	
\end{lemma}
Recall that (i) of Lemma \ref{basic est} was obtained first in \cite{CKP12}.

\begin{lemma}[Lemma 2.6, \cite{BBD20}]  \label{kernel est}
	Let $\varphi \in \mathscr S (\mathbb R)$ be an even function. Then for any $W >0$, there exists $C>0$ such that
	\[
	\Big| K _ { \varphi ( \alpha \sqrt{L}) }  (x,y) \Big| \le \frac{C}{V(x  \vee y , \alpha  )}\bigg(  1+ \frac{\rho(x,y) }{\alpha} \bigg)^{-W}
	\]
	for all $\alpha >0$ and $x,y \in X$.
\end{lemma}

\begin{lemma}[Proposition 2.10, \cite{BBD20}] \label{conv in S infty}
	Let $\psi $ be a partition of unity. Then for any $f\in \mathcal S_\infty '$, we have
	\[
	f = \sum_{j \in \mathbb Z} \psi_j \big(\sqrt{L}\big) (f) \;   \mathrm{in} \; \mathcal S_\infty '.
	\]
\end{lemma}

We will often use the following inequality  and may use these in the sequel without
stating any reasons.
For all $x,y, z \in X$ and all $\alpha,W >0$, we have
\begin{equation*}
	\bigg(1 + \frac{\rho(x,y)}{\alpha} \bigg)^{-W} 	\bigg(1 + \frac{\rho(y,z)}{\alpha} \bigg)^{-W} \lesssim  	\bigg(1 + \frac{\rho(x,z)}{\alpha} \bigg)^{-W}.
\end{equation*}

 Let $ 0< q\le \infty$, and $\omega \in A_\infty $. Let $0<p<t<r<\infty$ 	or $0<p\le t<r=\infty$.	
For convenience, we define sequence valued Bourgain-Morrey spaces  by 
\begin{align*}
	\| \{ g_j\}_{j\in \mathbb Z}  \| _{M_{p,\omega}^{t,r} (\ell^q)} := \Big\|  \| \{ g_j \}_{j\in \mathbb Z} \| _{\ell^q} \Big\| _{M_{p,\omega}^{t,r} },   
\end{align*}
and 
\begin{align*}
	\| \{ g_j\}_{j\in \mathbb Z}  \| _{  \ell ^q ( M_{p,\omega}^{t,r} ) } := \Big\|    \big  \{ \|  g_j \| _{   M_{p,\omega}^{t,r}   }  \big \}_{j\in \mathbb Z}  \Big\| _{  \ell^q },
\end{align*}
with finite norm respectively, where $\{ g_j \}_{j\in \mathbb Z}$  is a  sequence of measurable functions on $X$. 

The following result comes from \cite[Theorem 3.2]{BX25}.
\begin{lemma}\label{M bour weight}
	Let $1<p\le t<r=\infty $ or  let  $1<p <t <r <\infty$.  Let $\omega \in A_p$.
	Let $\beta = 1-  (1- \alpha_0)^p   / [\omega]_{A_p}$ where $\alpha_0$  is the reverse doubling  constant of $X$. If $ r >  -nt / \log_2 \beta  $, then 	
	
	{\rm(i) }  the Hardy-Littlewood maximal operator $\mathcal M$ is bounded on $M_{p,\omega}^{t,r}$;
	
	{\rm 	(ii)  }  for all sequences $\{f_j\} _{j \in \mathbb Z}  \in M_{p,\omega}^{t,r} (\ell^u)$, $1<u\le \infty $,   we have
	\[
	\|  \{  \mathcal M f_j \}_{j\in \mathbb Z}  \|_{  M_{p,\omega}^{t,r} (\ell^u) } \lesssim 	\|  \{  f_j \}_{j \in \mathbb Z}\|_{  M_{p,\omega}^{t,r} (\ell^u) }.
	\]
\end{lemma}

\section{Predual of weighted Bourgain-Morrey space} \label{predual BM}
In this section, we study the predual of weighted Bourgain-Morrey space on $X$.
First recall the concept of $(p,t) $-blocks; see for example, \cite[p. 32]{BRV99} in the Euclidean space case.
\begin{definition}
	Let $1< p \le t \le \infty $.  Let $\omega $ and $\omega' := \omega^{-1/ (p-1)} $ be  weights.   A function $b \in L^{p'} (\omega' ) $ is called a $ (p', t', \omega ') $-block if there exists a dyadic cube $Q \in \D$ such that supp $b \subset Q$ and 
	\begin{equation*}
	\|b\|_{L^{p'} (\omega') } : = \left( \int_{Q} |b(x)|^{p'} \omega^{-1 / (p-1)} (x) \d \mu (x)   \right)^{1/p'} \le \omega (Q)^{ 1/p' - 1/t' } =  \omega (Q)^{ 1/t - 1/p }.
	\end{equation*}
Let $1\le q' \le \infty $.
Similarly, a vector valued function $\vec b \in L^{p'} (\omega' ,\ell^{q'}  ) $ is called a $ (p', t', \omega ', \ell^{q'} ) $-block if there exists a dyadic cube $Q \in \D$ such that supp $\vec b \subset Q$ and 
\begin{equation*}
	\|\vec b\|_{L^{p'} (\omega', \ell^{q'} ) } : = \left( \int_{Q} \|\vec b(x)\|_{ \ell^{q'} }^{p'} \omega^{-1 / (p-1)} (x) \d \mu (x)   \right)^{1/p'} \le \omega (Q)^{ 1/p' - 1/t' } =  \omega (Q)^{ 1/t - 1/p }.
\end{equation*}
\end{definition}
Note that if $\omega \in A_\infty$, then $\omega '  \in (0,\infty) $ and is  a locally integrable function on $X$. Hence $\omega '$ is a weight.
\begin{definition}
			Let $1 < p \le t \le r \le  \infty $.   Let $\omega $ and $\omega' := \omega^{-1/ (p-1)} $ be  weights.  The block space $\H _{p', \omega '}^{t',  r' }   $ is the set of all $\mu$-measurable functions $f$ on $X$ which can be decomposed as 
		\begin{equation}\label{block decom}
			f = \sum_{k \in \mathbb Z} \sum_{\tau \in I_k}  \lambda_{k,\tau } b_{k,\tau }
		\end{equation}
	$\mu$-almost everywhere on $X$, where $ \{ \lambda_{k,\tau }  \}_{ k \in \mathbb Z, \tau \in I_k}  \in \ell^{r'} $  and, for each $k \in \mathbb Z $ and $\tau \in I_k  $, $b_{k,\tau }$ is a $(p', t', \omega ')$-block supported in $ Q^k_\tau$. Moreover, the norm of  $\H _{p', \omega' }^{t',  r' }   $  is defined by 
	\begin{equation*}
		\| f\|_{ \H _{p', \omega ' }^{t',  r' }   } : =\inf \left\{  \|\{ \lambda_{k,\tau }  \}_{k \in \mathbb Z, \tau \in I_k }\|_{ \ell^{r'}  }   \right\}
	\end{equation*}
where the infimum is taken over all decompositions of $f$ as in (\ref{block decom}).

Similarly, 	let $1 < p \le t \le r \le  \infty $ and $1\le q \le \infty $.   Let $\omega $ and $\omega' := \omega^{-1/ (p-1)} $ be  weights.  The block space $\H _{p', \omega '}^{t',  r' }  (\ell^{q'})  $ is the set of all $\mu$-measurable sequence valued functions $\vec f$ on $X$ which can be decomposed as 
\begin{equation}\label{block decom seq}
\vec 	f = \sum_{k \in \mathbb Z} \sum_{\tau \in I_k}  \lambda_{k,\tau } \vec  b_{k,\tau }
\end{equation}
$\mu$-almost everywhere on $X$, where $ \{ \lambda_{k,\tau }  \}_{ k \in \mathbb Z, \tau \in I_k}  \in \ell^{r'} $  and, for each $k \in \mathbb Z $ and $\tau \in I_k  $, $ \vec b_{k,\tau }$ is a $(p', t', \omega ', \ell^{q'})$-block supported in $ Q^k_\tau$. Moreover, the norm of  $\H _{p', \omega' }^{t',  r' }  (\ell^{q'})   $  is defined by 
\begin{equation*}
	\| f\|_{ \H _{p', \omega ' }^{t',  r' }  (\ell^{q'}) } : =\inf \left\{  \|\{ \lambda_{k,\tau }  \}_{k \in \mathbb Z, \tau \in I_k }\|_{ \ell^{r'}  }   \right\}
\end{equation*}
where the infimum is taken over all decompositions of $f$ as in (\ref{block decom seq}).
\end{definition}

We first consider the convergence of the series in (\ref{block decom}) and (\ref{block decom seq}).

\begin{theorem}
	Let $\omega \in A_\infty$. Let  $\beta_\omega : =  1 - (1- \alpha_0 ) ^{ q_\omega + A }  / [\omega ]_{A_  {q_\omega +A} }$ where $\alpha_0$  is the reverse doubling  constant of $X$ and $A >0$. 
	Let $ 1\le q \le \infty $.
	Let $1< p <t <\min \{ r,    ( -\log_2 \beta_\omega ) r /n \}  <\infty$ or $ 1< p <t <r =\infty $.
	
	{\rm (i)} 
	Assume that $ \{ \lambda_{k,\tau }  \}_{ k \in \mathbb Z, \tau \in I_k}  \in \ell^{r'} $  and, for each $k \in \mathbb Z $ and $\tau \in I_k  $, $\vec b_{k,\tau }$ is a $(p', t', \omega ' , \ell^{q'})$-block supported in $ Q^k_\tau$.
	Then the summation (\ref{block decom seq}) converges both  $\mu$-almost everywhere on $X$  and in  $L^1_{\rm loc} (\ell^{q'}) $ where $L^1_{\rm loc} (\ell^{q'})  $	is the $\ell^{q'}$ valued locally integrable function space.
	
		{\rm (ii)} 
	Assume that $ \{ \lambda_{k,\tau }  \}_{ k \in \mathbb Z, \tau \in I_k}  \in \ell^{r'} $  and, for each $k \in \mathbb Z $ and $\tau \in I_k  $, $b_{k,\tau }$ is a $(p', t', \omega ' )$-block supported in $ Q^k_\tau$.
	Then the summation (\ref{block decom}) converges both $\mu$-almost everywhere on $X$  and in  $L^1_{\rm loc}  $.
\end{theorem}

\begin{proof}
	We only prove (i) since (ii) is similar.
	We first claim that, for any given dyadic cube $Q^{k_0} _{\tau_0}\in \D$ with $ k_0 \in \mathbb Z , \tau_0 \in I_{k_0}$,
	\begin{equation} \label{abs integral}
		\int_{ Q^{k_0} _{\tau_0} } \sum_{k \in \mathbb Z} \sum_{\tau \in I_k}  |\lambda_{k,\tau } | \| \vec  b_{k,\tau } \|_{ \ell^{q'} } \d \mu (x) <\infty.
	\end{equation}
	In fact, by Tonelli's theorem, we have
	\begin{align*}
		\int_{ Q^{k_0} _{\tau_0} \in \D } \sum_{k \in \mathbb Z} \sum_{\tau \in I_k}  |\lambda_{k,\tau } | \| \vec  b_{k,\tau } \|_{ \ell^{q'} } \d \mu (x)  = \left( \sum_{ k = -\infty} ^{k_0 -1} +  \sum_{ k = k_0 } ^{\infty} 	 \right)  \sum_{\tau \in I_k}  \int_{ Q^{k_0} _{\tau_0}} |\lambda_{k,\tau } | \| \vec  b_{k,\tau } \|_{ \ell^{q'} } \d \mu (x) 
		:= S_1 +S_2.
	\end{align*}
For $S_1 $,  there is only one $ \tau_k  $ such that $Q ^{k_0}_{\tau_0}  \subset Q ^k _{\tau_k} $.
Since $ \omega \in A_{ q_\omega +A } $,  we have
\begin{equation*}
	\omega ( Q ^{k_0}_{\tau_0}  ) 	  \lesssim \beta_\omega  ^{ k_0 - k } \omega ( Q ^{k }_{\tau_k}  ),
\end{equation*} 
for every $ Q ^{k_0}_{\tau_0}   \subset Q ^{k }_{\tau_k}   $. Since $1/t -1/p <0 $, we have
\begin{equation*}
	\omega ( Q ^{k }_{\tau_k}  ) ^{1/t -1/p }  \lesssim   \beta_\omega  ^{ (k_0 - k)  (1/t-1/p ) }  \omega ( Q ^{k_0}_{\tau_0}  ) ^{1/t -1/p}.
\end{equation*}
By H\"older's inequality and $\beta_\omega \in (0,1)$, we obtain
\begin{align*}
S_1 =	& \sum_{ k = -\infty} ^{k_0 -1} \sum_{\tau \in I_k}  \int_{ Q^{k_0} _{\tau_0}} |\lambda_{k,\tau } | \| \vec  b_{k,\tau } \|_{ \ell^{q'} } \d \mu (x)   \\
	& \le \| \lambda \|_{  \ell^{r'} } \sum_{ k = -\infty} ^{k_0 -1}    \left\|  \| \vec  b_{k,{\tau_k} } \|_{ \ell^{q'} } \right\|_{ L^{ p ' }(\omega ') }  \omega (  Q ^{k_0}_{\tau_0} ) ^{ 1/p } \\
	&  \le \| \lambda \|_{  \ell^{r'} }\sum_{ k = -\infty} ^{k_0 -1} \omega (  Q^k_{\tau_k}  ) ^{ 1/t -1/p }  \omega (  Q ^{k_0}_{\tau_0} ) ^{ 1/p }  \\
	& \lesssim  \| \lambda \|_{  \ell^{r'} } \sum_{ k = -\infty} ^{k_0 -1}   \beta_\omega  ^{ (k_0 - k)  (1/t-1/p ) } \omega ( Q ^{k_0}_{\tau_0}  ) ^{1/t }  \\
	&\lesssim  \| \lambda \|_{  \ell^{r'} }   \omega ( Q ^{k_0}_{\tau_0}  ) ^{1/t } . 
\end{align*}
For $S_2$, 	as in Remark \ref{dya cube}, there are at most $c2^{(  k-k_0 )  n}$ dyadic cubes $ Q ^{k }_\tau    $ such that $Q ^{k}_{\tau}  \subset Q ^{k_0}_{\tau_0}  $ where $c$ is a constant depending only on $X$.
Since $ \omega \in A_{ q_\omega +A } $,  we have
\begin{equation*}
	\omega ( Q ^{k }_\tau  )  \lesssim \beta_\omega  ^{ k- k_0 } \omega ( Q ^{k_0}_{\tau_0}  ) ,
\end{equation*} 
for every $ Q ^{k }_\tau  \subset  Q ^{k_0}_{\tau_0}  $.
Let $ n /r + ( \log_2 \beta_\omega ) /t <0 $.
By H\"older's inequality, we obtain
\begin{align*}
	S_2 & =  \sum_{ k = k_0 } ^{\infty}  \sum_{\tau \in I_k , Q_\tau ^k  \subset Q^{k_0}_{\tau_0}} \int_{ Q_\tau ^k } |\lambda_{k,\tau } | \| \vec  b_{k,\tau } \|_{ \ell^{q'} } \d \mu (x) \\
	& \lesssim \sum_{ k = k_0 } ^{\infty}  \sum_{\tau \in I_k , Q_\tau ^k  \subset Q^{k_0}_{\tau_0}}|\lambda_{k,\tau } | \left\|  \| \vec  b_{k,\tau } \|_{ \ell^{q'} } \right\|_{ L^{ p ' }(\omega ') }  \omega (  Q ^{k}_{\tau} ) ^{ 1/p } \\
	&\lesssim  \sum_{ k = k_0 } ^{\infty}  \sum_{\tau \in I_k , Q_\tau ^k  \subset Q^{k_0}_{\tau_0}}|\lambda_{k,\tau } |  \omega (  Q ^{k}_{\tau} ) ^{ 1/t }  \\
		&\lesssim  \sum_{ k = k_0 } ^{\infty}  \beta_\omega  ^{ ( k- k_0 ) /t  } \omega ( Q ^{k_0}_{\tau_0}  )^{1/t}   \sum_{\tau \in I_k , Q_\tau ^k  \subset Q^{k_0}_{\tau_0}}|\lambda_{k,\tau } |    \\
				&\lesssim \omega ( Q ^{k_0}_{\tau_0}  )^{1/t}  \sum_{ k = k_0 } ^{\infty}  \beta_\omega  ^{ ( k- k_0 ) /t  }   2^{ (k-k_0) n /r }   \| \lambda \|_{\ell^{r'}} \\
	& \lesssim \| \lambda \|_{  \ell^{r'} }   \omega ( Q ^{k_0}_{\tau_0}  ) ^{1/t }   .
\end{align*}
Hence we prove (\ref{abs integral}). Then for any $Q \in \D$, we have
\begin{equation} \label{conver by abs inte}
	  \left\|  	\int_{ Q^{k_0} _{\tau_0} } \sum_{k \in \mathbb Z} \sum_{\tau \in I_k}  \lambda_{k,\tau }   \vec  b_{k,\tau } \d \mu (x) \right\|_{ \ell^{q'} }	 \le 	\int_{ Q^{k_0} _{\tau_0} } \sum_{k \in \mathbb Z} \sum_{\tau \in I_k}  |\lambda_{k,\tau } | \| \vec  b_{k,\tau } \|_{ \ell^{q'} } \d \mu (x) <\infty,
\end{equation}
which, together with the arbitrariness of $Q \in \D$  and (i)  of Lemma \ref{cube}, further implies
that the series in (\ref{block decom seq}) converges  in $\ell^{q'}$ $\mu$-almost everywhere on $X$.

The following claim comes from \cite[Proof of Theorem 2.8]{ZYY26}. For any given $k\in \mathbb Z$ and $\tau \in I_k  $, there exists a finite set $ \Lambda_{k,\tau}  \subset I_k $  of indices such that for any $y \in Q^k_\tau$  and $ R \in  ( \gamma ^k, \gamma ^{k - 1}  ) $, $ B (y,  R) \subset  \cup_{ \beta \in  \Lambda_{k,\tau} }  Q^k_\beta $. (Note that in Remark \ref{dya cube}, we assume that $\gamma = 1/2$ without loss of generality.)

From this claim and (\ref{conver by abs inte}),
we also infer that the series in (\ref{block decom seq}) converges in $L^1_{\rm loc} (\ell^{q'}) $.
\end{proof}

\begin{lemma}[Corollary 1.3.22, \cite{HNVW16}]
	 \label{dual banach}
	Let $(X, \mathcal A,\mu)$ be a $\sigma$-finite measure space and let $Y$ be
	reflexive or $Y^\ast$ be separable. Then for all $1\le p <\infty $ we have an isometric 	isomorphism
\begin{equation*}
	(L^p (S ,Y) )^\ast =  L^{p'} (S,Y^\ast)  .
\end{equation*}
	The $\sigma$-finiteness assumption is redundant for $1 < p < \infty.$
\end{lemma}

Denote by $L^{p} (\omega, \ell^q) $ the weighted $\ell^q$-valued Lebesuge function spaces such that 
\begin{equation*}
	\|f\|_{ L^{p} (\omega, \ell^q) } := \left(  \int_X |f(x)|^p \omega (x) \d \mu (x) \right)^{1/p} <\infty.
\end{equation*}
When $\omega \equiv 1$, we drop it.

\begin{lemma} \label{dual Lp ell q}
	Let $ \omega$ and $\omega ' := \omega^{ -1/ (p-1) } $ be weights. Then for $ 1< p,q <\infty $, we have an  isometric 	isomorphism
	\begin{equation*}
		(L^{p '} (\omega ' , \ell ^{q ' }) )^\ast =  L^{p} (\omega, \ell^q)  
	\end{equation*}
 in the following sense:

 {\rm (i)} for any $\vec g \in L^{p} (\omega, \ell^q)$, the linear functional $J_{\vec g} $, defined by setting, for any $\vec f \in L^{p '} (\omega ' , \ell ^{q ' }) $, $J_{\vec g} (\vec f) := \int_X  \sum_j f_j (x) g_j (x) \d \mu  (x) $, is bounded on $ L^{p '} (\omega ' , \ell ^{q ' }) $;
 
 {\rm (ii)} for any $J \in 	(L^{p '} ( \omega',  \ell ^{q ' }) )^\ast$, there exists a unique $\vec  g \in L^{p} (\omega, \ell^q)$ such that $J=J_{\vec g}$ in $	(L^{p '} ( \omega',  \ell ^{q ' }) )^\ast$; moreover,  $\|g\|_ { L^{p} (\omega, \ell^q)}  = \| J\|_{ (L^{p '} ( \omega',  \ell ^{q ' }) )^\ast }$. 
 
\end{lemma}
\begin{proof}
	(i) Using H\"older's inequality, we obtain (i).
	
	(ii)
	Since $\ell^q $ is separable, by Lemma \ref{dual banach}, we have
	\begin{equation} \label{dual L p ell q}
		(L^{p '} ( \ell ^{q ' }) )^\ast =  L^{p} ( \ell^q) .
	\end{equation}
Now let $ J \in 	(L^{p '} ( \omega',  \ell ^{q ' }) )^\ast$. 
For a function $\vec f\in L^{p '} ( \omega',  \ell ^{q ' })  $, $ \vec  f  (  \omega ^\prime)^{ 1/ p' }  \in  L^{p '} ( \ell ^{q ' })$. 
Define $J '  \in  (L^{p '} ( \ell ^{q ' }) )^\ast   $ by 
\begin{equation*}
	J'  (  \vec f  (  \omega ^\prime)^{ 1/ p' } ) := J (\vec f) .
\end{equation*}
Then by  (\ref{dual L p ell q}), there exists a unique $\vec g \in  L^{p} ( \ell^q)$ such that 
\begin{align*}
		J'  (  f  (  \omega ^\prime)^{ 1/ p' } ) & = \int_X \sum_{j \in \mathbb Z} (  \omega ^\prime  (x) )^{ 1/ p' } f_j (x) g_j (x) \d \mu (x) \\
		& = \int_X \sum_{j \in \mathbb Z}  f_j (x) g_j (x) \omega ^{-1/p} (x) \d \mu (x),
\end{align*}
and $\| \vec g \|_{L^{p} ( \ell^q)  } = \| J' \|_{  (L^{p '} ( \ell ^{q ' }) )^\ast } $.
Set  $ \vec h : =   \vec g   \omega ^{-1/p} \in L^{p} (\omega, \ell ^q)$. Then 
\begin{equation*}
	 J (\vec f) = 	J'  (  \vec f  (  \omega ^\prime)^{ 1/ p' } ) = \int_X \sum_{j \in \mathbb Z}  f_j (x) h_j (x)\d \mu (x) 
\end{equation*}
Note that 
\begin{equation*}
	\| J\|_{ 	(L^{p '} ( \omega',  \ell ^{q ' }) )^\ast } = \| J' \|_{  (L^{p '} ( \ell ^{q ' }) )^\ast }.
\end{equation*}
Hence we obtain 
\begin{equation*}
		\| J\|_{ 	(L^{p '} ( \omega',  \ell ^{q ' }) )^\ast } = \| \vec g \|_{L^{p} ( \ell^q)  } = \| \vec h  \|_{ L^{p} (\omega, \ell ^q) } .
\end{equation*}
Thus we prove (ii) and the proof is complete.
\end{proof}
Now we are ready to obtain the predual of weighted Bourgain-Morrey spaces.
\begin{theorem} \label{predual vector}
	Let $1 <q <\infty $. Let $\omega \in A_p$ for $p \in (1,\infty)$. Let $\beta_\omega : =  1 - (1- \alpha_0 ) ^{ p }  / [\omega ]_{A_  {p} }$ where $\alpha_0$  is the reverse doubling  constant of $X$. 
Let $ 1< p< t < \min\{ r, \log_2 ( \beta_\omega ^{-1} )  r  /n \}  <\infty$ or let $ 1 <p \le t <r =\infty $.
	Then the dual of $  \H _{p', \omega '  }^{t',  r' } (\ell^{q'}) $ is $ M_{p,\omega}^{t,r} (\ell^q ) $ in the following sense:
	
	{\rm (i)} for any $\vec g \in M_{p,\omega}^{t,r} (\ell^q )$, the linear functional $J_{\vec g} $, defined by setting, for any $\vec f \in \H _{p', \omega '  }^{t',  r' } (\ell^{q'})$, $J_{\vec g} (\vec f) := \int_X  \sum_{j \in \mathbb Z} f_j (x) g_j (x) \d \mu  (x) $, is bounded on $ \H _{p', \omega '  }^{t',  r' } (\ell^{q'})$;
	
		{\rm (ii)} for any $J \in \left(  \H _{p', \omega '  }^{t',  r' }  (\ell^{q'}) \right)^\ast$, there exists a unique $\vec  g \in M_{p,\omega}^{t,r}(\ell^q) $ such that $J=J_{\vec g}$ in $\left(  \H _{p', \omega '  }^{t',  r' } (\ell^{q'}) \right)^\ast$; moreover,  $\|g\|_ { M_{p,\omega}^{t,r} (\ell^q )}  = \| J\|_{ \left(  \H _{p', \omega '  }^{t',  r' } (\ell^{q'}) \right)^\ast }$. 
\end{theorem}
\begin{proof}
	We first prove (i). 
	Let $\vec g \in M_{p,\omega}^{t,r} (\ell^q) $. For any $\vec f \in  \H _{p', \omega '  }^{t',  r' } (\ell^{q'})$, there exist a sequence $\{ \lambda_{k,\tau } \}_{ k \in \mathbb Z,  \tau \in I_k }  \in \ell^{r'}$ and a sequence $\{\vec  b_{k,\tau} \}_{ k \in \mathbb Z,  \tau \in I_k } $ of $(p',t', \omega ', \ell^{q'})$-blocks  such that for $\mu$-almost every $x \in X$, $\vec f (x) =  \sum_{k \in \mathbb Z} \sum_{\tau \in I_k}  \lambda_{k,\tau } \vec  b_{k,\tau }$ and $ \| \{ \lambda_{k,\tau } \}_{ k \in \mathbb Z,  \tau \in I_k }  \|_{ \ell^{r'} } \le (1+\epsilon)  \|\vec f\|_{ \H _{p', \omega ' }^{t',  r' }  (\ell^{q'}) } $. By H\"older's inequality,
	\begin{align*}
		|J_{\vec g} (\vec f) | & \le \int_X \sum_{j \in \mathbb Z} | f_j (x)  g_j(x)|  \d \mu  (x)  \\
		& \le \sum_{k \in \mathbb Z} \sum_{\tau \in I_k}  | \lambda_{k,\tau } | \int_{Q^k_\tau } \sum_{j \in \mathbb Z}   | b_{k,\tau }^j |  | g_j (x)|  \d \mu (x) \\
		& \le  \sum_{k \in \mathbb Z} \sum_{\tau \in I_k}  | \lambda_{k,\tau } | 
		\left( \int_{Q^k_\tau }   \left( \sum_{j \in \mathbb Z} | g_j (x)|^q \right)^{p/q}  \omega  (x) \d \mu (x)  \right)^{1/p} \left( \int_{Q^k_\tau }  \left( \sum_{j \in \mathbb Z}   | b_{k,\tau }^j |^{q'} \right)^{p'/{q'}}  \omega ^{- p'/p} (x) \d \mu (x)  \right)^{1/p'}
		\\
		& \le  \sum_{k \in \mathbb Z} \sum_{\tau \in I_k}  | \lambda_{k,\tau } | 
		\left( \int_{Q^k_\tau }  \left( \sum_{j \in \mathbb Z} | g_j (x)|^q \right)^{p/q}  \omega  (x) \d \mu (x)  \right)^{1/p} \omega (Q)^{ 1/t - 1/p } \\
		& \le  \left(  \sum_{k \in \mathbb Z} \sum_{\tau \in I_k}  | \lambda_{k,\tau } |^{r'} \right) ^{1/r'}
		\| \vec g\|_{ M_{p,\omega}^{t,r} (\ell^q ) } \\
		& \le (1+ \epsilon ) \|\vec  f\|_{ \H _{p', \omega ' }^{t',  r' } (\ell^{q'}) }  \|\vec g\|_{ M_{p,\omega}^{t,r} (\ell^q) } .
	\end{align*}
Letting $\epsilon\to 0^+$ and taking the  supremum over $ \|\vec  f\|_{ \H _{p', \omega ' }^{t',  r' } (\ell^{q'}) }  \le 1$, we obtain $  \| J_{\vec g} \|_{ \left(  \H _{p', \omega '  }^{t',  r' } (\ell^{q'}) \right)^\ast } \le \|g\|_ { M_{p,\omega}^{t,r} (\ell^q )} $.

	Next we prove (ii). Let $J \in \left(  \H _{p', \omega '  }^{t',  r' }  (\ell^{q'}) \right)^\ast $. We divide the proof into the following four
	steps.
	
	Step 1. 
	Note that for any $k \in \mathbb Z $ and $\tau \in I_k$ and for any  $\vec  f  \in L^{p'} (\omega ',  Q_\tau^k, \ell^{q'} )$ with $ \| \vec  f   \|_{ L^{p'} (\omega ',  Q_\tau^k, \ell^{q'} ) } : = \left( \int_{ Q^k_\tau } \| \vec f (x) \|_{\ell^{q'}} ^{p'} \omega ' (x) \d \mu (x)    \right) ^{1/p'} \neq 0  $, 
	\begin{equation*}
		\frac{\omega (Q_\tau^k )^{1/t-1/p}   }{  \| \vec  f   \|_{ L^{p'} (\omega ',  Q_\tau^k, \ell^{q'} ) } } \vec f \chi_{ Q_\tau^k }
	\end{equation*}
	is a $ (p', t', \omega ', \ell^{q'})$-block. Thus, the functional $J_{k,\tau}  $, defined by $ J_{k,\tau} (\vec  f) := J ( f \chi_{ Q_\tau^k }  ) $ for any $\vec  f  \in L^{p'} (\omega ',  Q_\tau^k , \ell^{q'}) $, belongs to $ \left(   L^{p'} (\omega ', \ell^{q'} ) \right) ^\ast   $. From this and $ p , q \in (1,\infty) $, by Lemma \ref{dual Lp ell q}, we deduce that there exists a unique function $\vec  g_{k,\tau}  \in L^p (\omega, Q_\tau^k ,\ell^q ) $  such that,  for any $ f  \in  L^{p'} (\omega ' ,Q_\tau^k, \ell^{q'} ) $,
	\begin{equation}\label{J Q}
		J_{k,\tau} (\vec f)  = \int_{ Q^k_\tau } \sum_{j \in \mathbb Z} f_j (x) g_{k,\tau}^j (x) \d \mu  (x).
	\end{equation}
	By the uniqueness of $ \{\vec  g_{k,\tau} \}_{k \in \mathbb Z, \tau \in I_k} $ and both (i) and (ii) of lemma \ref{cube}, there exists a unique function $\vec  g \in L^p_{\rm loc} (\omega, \ell^q)$ such that, for any $k, \ell \in \mathbb Z$, $ \tau \in I_k $  and $\beta \in I_\ell$ satisfying $ Q^k_\tau \subset Q^\ell_\beta $, $\vec g = \vec g_{ k,\tau } = \vec g_{\ell , \beta} $  $\mu$-almost everywhere on $Q^k_\tau$. From this and $ (\ref{J Q}) $, we get that, for any $\vec  f \in  L^{p'} (\omega ',  Q_\tau^k, \ell^{q'} )$ with  $k \in \mathbb Z$ and $ \tau \in I_k $,
	\begin{equation}\label{J f Q}
		J(\vec f \chi_{ Q^k_\tau })  = J_{k,\tau} ( \vec f) =\int_{ Q^k_\tau } \sum_{j \in \mathbb Z} f_j (x) g_{k,\tau}^j (x) \d \mu  (x)=\int_{ Q^k_\tau }\sum_{j \in \mathbb Z} f_j (x) g_j (x) \d \mu  (x).
	\end{equation} 
	Step 2. In this step, we show $\vec g \in  M_{p,\omega}^{t,r} (\ell^q)$ with $ \| \vec g\|_{ M_{p,\omega}^{t,r} (\ell^q)} \le \|J\|_{ \left(  \H _{p', \omega ' }^{t',  r' } (\ell^{q'}) \right )^\ast   } $. For each $j,k\in \mathbb Z $, $\tau \in I_k$, and $x\in X$, let
	\begin{equation} \label{block b g}
		b_{k,\tau}^j := \begin{cases} \omega (Q^k_\tau ) ^{1/t-1/p}  \| \vec g\|_{ L^p (\omega, Q^k_\tau , \ell^{q}) } ^{1-p}   |g_j(x)|^{q-1} \omega (x)  \| \vec g(x)\|_{\ell^q}^{(p-q)}    \chi_{ Q^k_\tau } (x) 
			& {\rm if}\; \| \vec  g\|_{ L^p (\omega, Q^k_\tau, \ell^q) } \neq 0, \\
			0, & {\rm if}\; \| \vec  g\|_{ L^p (\omega, Q^k_\tau, \ell^q) } = 0.
		\end{cases}
	\end{equation} 
	Then for $ \| \vec g\|_{ L^p (\omega, Q^k_\tau , \ell^{q}) } \neq 0 $,
	\begin{align*}
		\| \vec b_{k,\tau} \|_{L^{p'} (\omega ' , Q^k_\tau , \ell^{q'})  } & = \omega (Q^k_\tau ) ^{1/t-1/p}  \| \vec g\|_{ L^p (\omega, Q^k_\tau , \ell^{q}) } ^{1-p} \left( \int_{ Q^k_\tau }  \left(  \sum_{j \in \mathbb Z}  |  g_j (x) |^{ (q-1) q'  } \right)^{p'/q'}  \| \vec g(x)\|_{\ell^q}^{ (p-q) p ' } \omega(x) \d \mu  (x)  \right)^{1/p'} \\
		& = \omega (Q^k_\tau ) ^{1/t-1/p}  .
	\end{align*}
	Hence $\vec b_{k,\tau} $ is a $ (p ', t', \omega ', \ell^{q'})$-block supported in $Q^k_\tau $. Fix $x_0 \in X$. For any $N \in \mathbb N$, let 
	\begin{equation*}
		\Omega_N := \left\{ (k,\tau) : |k|\le N, \alpha \in I_k,   Q^k_\tau \subset B ( x_0,N)   \right\}.
	\end{equation*}
	By (\ref{double}) and Remark \ref{dya cube}, we have for each $N \in \mathbb N$, $\sharp	\Omega_N <\infty $. From this, (\ref{block b g}), (\ref{J f Q}) and $J \in \left(  \H _{p', \omega '  }^{t',  r' }  (\ell^{q'}) \right)^\ast $, we obtain, for any nonnegative sequence $\{ \xi_{k,\tau}\}_{ k\in \mathbb Z, \tau \in I_k} \in \ell^{r'}  $ with $ \| \{ \xi_{k,\tau}\}_{ k\in \mathbb Z, \tau \in I_k} \|_{ \ell ^{r'}} \le 1  $,
	\begin{align*}
		\sum_{ (k,\tau ) \in \Omega_N  } \xi_{k,\tau} \omega (Q_\tau ^k  ) ^{1/t-1/p} \| \vec g \|_{L^p (\omega ,Q_\tau ^k ,\ell^q )}  
		&= \sum_{ (k,\tau ) \in \Omega_N  }  \xi_{k,\tau}  \int_{Q^k_\tau } \sum_{j \in \mathbb Z}  b_{k,\tau}^j (x) |g_j(x) | \d \mu  (x) \\
		& = \sum_{ (k,\tau ) \in \Omega_N  }  \xi_{k,\tau}  J (  \{ b_{k,\tau}^j \sgn g_j  \}_{j \in \mathbb Z}  ) \\
		& = J  \left( \sum_{ (k,\tau ) \in \Omega_N  } \xi_{k,\tau}  \{ b_{k,\tau}^j \sgn g_j  \}_{j \in \mathbb Z}  \right)  \\
		& \le \|J\|_{ (  \H _{p', \omega ' }^{t',  r' }  (\ell^{q'} ) )^\ast    }  \left\| \sum_{ (k,\tau ) \in \Omega_N  } \xi_{k,\tau}  \{ b_{k,\tau}^j \sgn g_j  \}_{j \in \mathbb Z}   \right\|_{ \H _{p', \omega ' }^{t',  r' } (\ell^{q'}) }  \\
		& \le \|J\|_{ (  \H _{p', \omega ' }^{t',  r' }  (\ell^{q'} ) )^\ast   }  \| \{ \xi_{k,\tau}\}_{ k\in \mathbb Z, \tau \in I_k} \|_{ \ell ^{r'}}  \\
		& \le \|J\|_{ (  \H _{p', \omega ' }^{t',  r' } (\ell^{q'} )  )^\ast   }.
	\end{align*}
	By this and the arbitrariness of $N \in \mathbb N$, we have
	\begin{align*}
		\| \vec  g\|_{ M_{p,\omega}^{t,r} (\ell^q ) } & = \bigg\|  \bigg\{\omega (Q_\tau ^k  ) ^{1/t-1/p}\bigg(\int_{Q_\tau ^k } \|\vec g(x)\|^{p}\omega(x) \mathrm {d} \mu(x)  \bigg)^{1/p}\bigg\}_{k \in\mathbb{Z}, \tau \in I_k}\bigg\|_{\ell^{r}} \\ 
		& = \sup_{ \| \{ \xi_{k,\tau}\}_{ k\in \mathbb Z, \tau \in I_k} \|_{ \ell ^{r'}} \le 1    } 	\sum_{ (k,\tau ) \in \Omega_N  } \xi_{k,\tau} \omega (Q_\tau ^k  ) ^{1/t-1/p} \| \vec g \|_{L^p (\omega ,Q_\tau ^k ,\ell^q )}    \\
		& \le \|J\|_{ (  \H _{p', \omega ' }^{t',  r' } (\ell^{q'}) )^\ast   } <\infty ,
	\end{align*}
	which further implies that $ \vec g \in M_{p,\omega}^{t,r} (\ell^q) $. By (i), we obtain $\| \vec  g\|_{ M_{p,\omega}^{t,r} (\ell^q ) } = \|J\|_{ (  \H _{p', \omega ' }^{t',  r' } (\ell^{q'}) )^\ast   }  $.
	
	Step 3. In this step, we show that $ J  = J_{\vec g}  \in (  \H _{p', \omega ' }^{t',  r' } (\ell^{q'}) )^\ast   $. Let $\vec f \in   \H _{p', \omega ' }^{t',  r' } (\ell^{q'})  $. Then there exist  a sequence $\{ \lambda_{k,\tau } \}_{ k \in \mathbb Z,  \tau \in I_k }  \in \ell^{r'}$ and a sequence $\{\vec b_{k,\tau} \}_{ k \in \mathbb Z,  \tau \in I_k } $ of $(p',t', \omega ', \ell^{q'})$-blocks  such that for $\mu$-almost every $x \in X$, $\vec f (x) =  \sum_{k \in \mathbb Z} \sum_{\tau \in I_k}  \lambda_{k,\tau } \vec b_{k,\tau }$ and $ \| \{ \lambda_{k,\tau } \}_{ k \in \mathbb Z,  \tau \in I_k }  \|_{ \ell^{r'} } \le (1+\epsilon)  \|\vec f\|_{ \H _{p', \omega ' }^{t',  r' } (\ell^{q' }) } $.
	For each  $N\in \mathbb N$, let $\vec f_N :=  \sum_{ (k,\tau) \in \Omega_N } \lambda_{k,\tau } \vec b_{k,\tau} $. 
	Since $ r' \in [1,\infty)$, we get
	\begin{equation} \label{f_N to f in H}
		\lim_{N \to \infty}\| \vec f-  \vec f_N \|_{\H _{p', \omega ' }^{t',  r' } (\ell^{q' }) } \le \lim_{N \to \infty} \left\| \{ \lambda_{k,\tau }  \}_{ (k,\tau ) \in  \complement \Omega_N  } \right\|_{\ell^{r'} } =0.
	\end{equation}
	Hence 
	\begin{equation*}
		\left| \lim_{N \to \infty} \int_X \sum_{j \in \mathbb Z}  (f_j-f_N^j ) (x) g_j (x) \d \mu (x)  \right| \le \lim_{N \to \infty}\| \vec f- \vec f_N \|_{\H _{p', \omega ' }^{t',  r' } (\ell^{q'}) } \|\vec g\|_{M_{p,\omega}^{t,r} (\ell^q ) } =0.
	\end{equation*}
	By this, (\ref{f_N to f in H}) and (\ref{J f Q}), we conclude 
	\begin{align*}
		J (\vec f) = \lim_{N \to \infty} J (\vec f_N) &=  \lim_{N \to \infty} \sum_{  (k,\tau) \in \Omega_N  } \lambda_{k,\tau } J ( \vec  b_{k,\tau} ) \\
		& =  \lim_{N \to \infty} \sum_{  (k,\tau) \in \Omega_N  } \lambda_{k,\tau }  \int_X   \sum_{j \in \mathbb Z}  b_{k,\tau}^j  g_j (x) \d \mu (x) \\
		& = \lim_{N \to \infty} \int_X \sum_{j \in \mathbb Z}  f_N^j (x) g _j (x) \d \mu (x) \\ 
		& = \int_X \sum_{j \in \mathbb Z} f_j (x) g_j (x) \d \mu (x)  = J_{\vec g} ( \vec f).
	\end{align*}
	Step 4. In this step, we prove the uniqueness of $\vec g$. Assume that $\vec { \tilde g} $ satisfies all the results in  the above three steps. Then by result in Step 3, we have that, for any $\vec f \in \H _{p', \omega ' }^{t',  r' } (\ell^{q' }) $,
	\begin{equation*}
		\int_X \sum_{j \in \mathbb Z} f_j (x) g_j (x) \d \mu  (x) = J (\vec f) = \int_X \sum_{j \in \mathbb Z} f_j (x) \tilde g_j (x) \d \mu  (x).
	\end{equation*}
Now fix $j \in \mathbb Z$.
	Note that for any given $Q \in \D$, 
	\begin{equation*}
		\{   \ldots, 0,  \overline{\sgn (g_j -\tilde g_j ) \chi_Q } , 0 ,\ldots   \}  \in  \H _{p', \omega ' }^{t',  r' } (\ell^{q' }) .
	\end{equation*}
	 Then we get
	\begin{equation*}
		\int_Q \overline{\sgn (g_j -\tilde g_j )  } (g_j  (x) - \tilde g_j (x) ) \d \mu (x) = \int_Q |g_j (x) - \tilde g_j (x) |\d \mu  (x) =0 .
	\end{equation*}
	By the arbitrariness of $Q \in \D$, $  g_j = \tilde g _j $ $\mu$-almost everywhere on $X$ and $\tilde g_j = g_j $ in $M_{p,\omega}^{t,r} $. 
	By arbitrariness  of $j \in \mathbb Z$, $ \vec g =\vec { \tilde g}  $ $\mu$-almost everywhere on $X$ and $\vec { \tilde g}  = \vec g$ in $M_{p,\omega}^{t,r} $. 
	
	Combining the above four steps, we finish the proof of (ii). Hence the proof of Theorem \ref{predual vector} is complete.
\end{proof}

Let $ \vec f = (\ldots, 0,f_0, 0, \ldots ) $ and $ \vec g = (\ldots, 0,g_0, 0, \ldots ) $.
Then we can obtain the following scalar version of Theorem \ref{predual vector}.

\begin{theorem}\label{predual scalar}
	 Let $\omega \in A_p$ for $p \in (1,\infty)$. Let $\beta_\omega : =  1 - (1- \alpha_0 ) ^{ p }  / [\omega ]_{A_  {p} }$ where $\alpha_0$  is the reverse doubling  constant of $X$. 
	Let $ 1< p< t < \min\{ r, \log_2 ( \beta_\omega ^{-1} )  r  /n \}  <\infty$ or let $ 1 <p \le t <r =\infty $.
	Then the dual of $  \H _{p', \omega '  }^{t',  r' }  $ is $ M_{p,\omega}^{t,r} $ in the following sense:
	
	(i) for any $ g \in M_{p,\omega}^{t,r}$, the linear functional $J_{ g} $, defined by setting, for any $ f \in \H _{p', \omega '  }^{t',  r' } $, $J_{ g} ( f) := \int_X   f (x) g (x) \d \mu  (x) $, is bounded on $ \H _{p', \omega '  }^{t',  r' } $;
	
	(ii) for any $J \in \left(  \H _{p', \omega '  }^{t',  r' }  \right)^\ast$, there exists a unique $ g \in M_{p,\omega}^{t,r}$ such that $J=J_{ g}$ in $\left(  \H _{p', \omega '  }^{t',  r' }  \right)^\ast$; moreover,  $\|g\|_ { M_{p,\omega}^{t,r} }  = \| J\|_{ \left(  \H _{p', \omega '  }^{t',  r' }  \right)^\ast }$. 
\end{theorem}

Next we study the properties of block space, which will be frequently used.
First we have the lattice property of block spaces $\H _{p', \omega '}^{t',  r' }   $.
\begin{lemma}
	Let $1 < p \le t \le r \le  \infty $.   Let $\omega $ and $\omega' := \omega^{-1/ (p-1)} $ be  weights.
	Then, a function $f$ belongs to $\H _{p', \omega ' }^{t',  r' }$ if and only if there exists $g \in \H _{p', \omega ' }^{t',  r' }$  such that $ |f(x)|\le g (x) $ for $\mu$ a.e. $x\in X$.
\end{lemma}
\begin{proof}
	First suppose that $f \in \H _{p', \omega ' }^{t',  r' }$. Then  decompose  $f$ as $f =\sum_{k \in \mathbb Z} \sum_{\tau \in I_k}  \lambda_{k,\tau }   b_{k,\tau } $ where $ \{ \lambda_{k,\tau } \}_{ k \in \mathbb Z,  \tau \in I_k } \in   \ell^{r'}  $  and $b_{k,\tau } $ is a $(p',t', \omega ')$-block. Let $g = \sum_{k \in \mathbb Z} \sum_{\tau \in I_k}  |\lambda_{k,\tau }|   |b_{k,\tau }|$. Then $|f (x)|\le g (x)  $ and $g \in  \H _{p', \omega ' }^{t',  r' } $.

	Now suppose that there exists $g \in \H _{p', \omega ' }^{t',  r' }$  such that $ |f(x)| \le g (x) $ for $\mu$ a.e. $x\in X$. Decompose $g$ as $g =\sum_{k \in \mathbb Z} \sum_{\tau \in I_k}  \lambda_{k,\tau }   b_{k,\tau } $ where $ \{ \lambda_{k,\tau } \}_{ k \in \mathbb Z,  \tau \in I_k } \in   \ell^{r'}  $  and $b_{k,\tau } $ is a $(p',t', \omega ')$-block. Then we see that 
	\begin{equation*}
		\chi_{\{  y \in X : g (y) \neq 0\} } (x) = \sum_{k \in \mathbb Z} \sum_{\tau \in I_k}  \lambda_{k,\tau }   b_{k,\tau } (x)  g(x)^{-1} ,
	\end{equation*}
	and, hence $ f (x) = \sum_{k \in \mathbb Z} \sum_{\tau \in I_k}   \lambda_{k,\tau }  \frac{f(x)}{g(x)}  b_{k,\tau } (x) $ for $\mu$-almost all $x\in X$. Since $\frac{|f(x)|}{g(x)} \le 1  $ $\mu$ a.e. $x\in X$, the function $(f/g) b_{k,\tau } $ is a $(p',t', \omega ')$-block. This proves the lemma.
\end{proof}

\begin{example} \label{exm ploy}
	Let $\omega \in A_p $ for $p \in (1,\infty)$. Let  $\beta_\omega : =  1 - (1- \alpha_0 ) ^p / [\omega ]_{A_p}$ where $\alpha_0$  is the reverse doubling  constant of $X$. 
	Let $ 1< p< t < \min\{ r, \log_2 ( \beta_\omega ^{-1} )  r  /n \}  <\infty$ or let $ 1 <p \le t <r =\infty $.	
	If $m > n + (\log_2 \beta_\omega  )  /t$, then 
	$\| ( 1+ \rho ( \cdot,x_0))^{-m} \|_{\H _{p', \omega '  }^{t',  r' }  } < \infty$. 
	Hence $\mathcal S \subset \H _{p', \omega '  }^{t',  r' } $.
\end{example}
\begin{proof}
	We consider balls since each ball can be covered by finite cubes.
	Let $ B_{ 2^\ell } : = B (x_0, 2^\ell ) $ with $\ell \in \mathbb \mathbb N_+$.
	\begin{align*}
		( 1+ \rho ( \cdot,x_0))^{-m}  =  ( 1+ \rho ( \cdot,x_0))^{-m} \chi_{B_{2} }  +  \sum_{ \ell =1 } ( 1+ \rho ( \cdot,x_0))^{-m}  \chi_{B_{ 2^{\ell+1}   } \backslash B_{ 2^\ell }   } .
	\end{align*}
	We only consider the second sum. Since $\omega \in A_p$, there exists $\beta_\omega : =  1 - (1- \alpha_0 ) ^p / [\omega ]_{A_p}  \in (0,1)$ such that for $ k\in \mathbb N_+ $,
	\begin{equation*}
		\omega (B_ {2 ^{\ell} }  )  \lesssim \beta_\omega ^k \omega ( B_{2 ^{\ell+k}  } ) .
	\end{equation*} 
	Then by (\ref{double}),
	\begin{align*}
		\| ( 1+ \rho ( \cdot,x_0))^{-m} \chi_{B_{ 2^{\ell+1}   } \backslash B_{ 2^\ell }   }  \|_{ L^{p'} (\omega ' )  } & = \left( \int_{ B_{ 2^{\ell+1}   } \backslash B_{ 2^\ell }  }   ( 1+ \rho ( y,x_0))^{-m p'}  \omega^{-1/ (p-1)}  \d \mu (y) \right)^{1/p'} \\
		& \lesssim  2^{-m\ell}  \left( \frac{V ( B_ {2^{\ell +1} } ) }{V ( B_{2^{\ell +1} } ) }  \int_{  B_{2^{\ell +1} }}\omega^{-1/ (p-1)}  \d \mu (y)\right)^{1/p'}  \\
		& \lesssim   2^{-m\ell} V ( B_{2^{\ell +1} } ) ^{1/p'} \Big(   \frac{1}{ V ( B_{2^{\ell +1} } ) } \int_{B_{2^{\ell +1} } } \omega(x) \mathrm {d} \mu(x)    \Big) ^{-1/p} \\
		& = 2^{-m\ell}  V ( B_{2^{\ell +1} } )  \omega (  B_{2 ^{\ell+1} } ) ^{ -1/p } \\
		& =2^{-m\ell}  V ( B_{2^{\ell +1} } )   \omega (  B_{2 ^{\ell+1} } ) ^{ -1/t }  \omega (  B_{2 ^{\ell+1} } ) ^{1/t -1/p }  \\
		&\lesssim 2^{-m\ell} 2^{\ell n} V (B_2) \beta_\omega ^{\ell /t}  \omega (  B_{2} ) ^{ -1/t }  \omega (  B_{2 ^{\ell+1} } ) ^{1/t -1/p }  \\
		& \lesssim   2^{-m\ell} 2^{\ell n} \beta_\omega ^{\ell /t}  \omega (  B_{2 ^{\ell+1} } ) ^{1/t -1/p } .
	\end{align*}
	Hence $ c 2^{m\ell} 2^{ - \ell n} \beta_\omega ^{ - \ell /t}  ( 1+ \rho ( \cdot,x_0))^{-m} \chi_{B_{ 2^{\ell+1}   } \backslash B_{ 2^\ell }   } $  is a  $(p' ,t' ,\omega ')$-block.
	Let $m > n + (\log_2 \beta_\omega  )  /t$ and we obtain
	\begin{equation*}
		\| ( 1+ \rho ( \cdot,x_0))^{-m} \|_{\H _{p', \omega '  }^{t',  r' }  } \lesssim  \left(\sum_{\ell =0}^\infty   \left|2^{-m\ell} 2^{\ell n} \beta_\omega ^{\ell /t} \right|^{r'}   \right)^{1/r'}  <\infty .
	\end{equation*}
	By this, we obtain $\mathcal S \subset \H _{p', \omega '  }^{t',  r' } $.
\end{proof}

Using Theorems \ref{predual vector}  and \ref{predual scalar}, we show the block spaces are Banach spaces. 

\begin{theorem} \label{block space is Banach}
	Let $1 <q <\infty $. Let $\omega \in A_p$ for $p \in (1,\infty)$. Let  $\beta_\omega : =  1 - (1- \alpha_0 ) ^p / [\omega ]_{A_p}$ where $\alpha_0$  is the reverse doubling  constant of $X$. 
	Let $ 1< p< t < \min\{ r, \log_2 ( \beta_\omega ^{-1} )  r  /n \}  <\infty$ or let $ 1 <p \le t <r =\infty $.
	Then $ \H _{p', \omega '  }^{t',  r' } (\ell^{q'}) $ and $ \H _{p', \omega '  }^{t',  r' }  $ are  Banach spaces. 
\end{theorem}

\begin{proof}
	We only  prove the completeness  since others are easy. 	
	Let $ \{\vec f ^{(i)} \}_{i\in \mathbb N_+}$  be a sequence such that for each $i \in \mathbb N_+$, $f_i \in \H _{p', \omega '  }^{t',  r' }  $, $ \sum_{i\ge 1} \| \vec f ^{(i)}  \|_{\H _{p', \omega '  }^{t',  r' } (\ell^{q'}  ) }  <\infty  $.
	From  Theorem \ref{predual scalar},
	for any ball $B$, we have
	\begin{equation*}
		\int_{B} \sum_{i\ge 1} \| \vec   f ^{(i)} (y) \|_{ \ell^{q'} } \d \mu (y) =\sum_{i\ge 1} \int_{B}  \| \vec   f ^{(i)} (y) \|_{ \ell^{q'} }  \d \mu (y) \le \| \chi_B \|_{  M _{p,\omega}^{t,r}   }  \sum_{i \ge 1} \|  \vec  f ^{(i)} \|_{ 	\H _{p', \omega '  }^{t',  r' } (\ell^{q'}) }.
	\end{equation*}
	From $\|  \sum_{i \ge 1}  f_i \|_{ \ell^{q'} }   \le \sum_{i \ge 1} \| f_i  \|_{\ell^{q'} }$,
	we obtain  $\vec f = \sum_{i \ge 1} \vec f ^{ (i)}  $ is a well defined, $\ell^{q'} $-valued $\mu$-measurable function and  $\vec f\in  L^1_{\operatorname{loc}} ( \ell^{q'} ) $. So we only need to show that  $\vec f =\sum_{i \ge 1} \vec f ^{ (i)}$ belongs to $\H_{p', \omega ' }^{t',r'}( \ell^{q'}  ) $.
	Fix $\epsilon > 0$. There exists a positive integer $N_0$ such that for all $N \ge N_0$, 
	\begin{equation*}
		\sum_{i \ge N } \|    \vec f^{(i)} \|_{ 	\H_{p', \omega ' }^{t',r'}(  \ell^{q'}  )  }  <\epsilon.
	\end{equation*}
	For this $\epsilon > 0$, there exists a decomposition 
	\begin{equation*}
		\vec f^{ (i)}  =   \sum_{k \in \mathbb Z, \tau \in I_k} \lambda_{i,k,\tau} \vec b_{i,k,\tau},
	\end{equation*}
	where each $\vec b_{i,k,\tau}$ is a  $(p',t',  \omega', \ell^{q'}  )$-block supported on $Q^k_\tau$ and 
	\begin{equation*}
		\left(    \sum_{ k \in \mathbb Z, \tau \in I_k  }   | \lambda_{i,k,\tau}| ^{r'} \right)^{1/r'}  \le (1+\epsilon)  \|    \vec f^{(i)} \|_{ 	\H_{p', \omega ' }^{t',r'}( \ell^{q'}  )  }.
	\end{equation*}
	Furthermore, for any $ 1\le i \le N_0 $, there exists an index set $ M _i  \subset Y := \{(j,\tau) : j\in \mathbb Z, \tau \in I_k  \} $  such that 
	\begin{equation*}
		\bigg\|  \vec f^{(i)}  - \sum_{(k,\tau)\in M_i} \lambda_{i,k,\tau} \vec b_{i,k,\tau}  \bigg\|_{ 	\H_{p', \omega ' }^{t',r'}( \ell^{q'}  )  } \le \bigg(  \sum_{(k,\tau)\in   Y  \backslash  M_i} |\lambda_{i,k,\tau}|^{r'} \bigg) ^{1/r'} \le 2^{-i} \epsilon.
	\end{equation*}
	Therefore, for all ball $B$, 
	\begin{align*}
		&	\int_B \bigg\|\vec  f(y)  -  \sum_{i=1}^{N_0}  \sum_{(k,\tau)\in M_i} \lambda_{i,k,\tau} \vec b_{i,k,\tau} (y) \bigg\|_{ \ell^{q'}  }  \d \mu (y) \\
		& \le \int_B  \bigg\|\vec  f(y) -\sum_{i=1}^{N_0} \vec f^{(i)} (y)  \bigg\|_{ \ell^{q'}  } \d \mu (y) + \int_B  \bigg\| \sum_{i=1}^{N_0} \vec f^{(i)} (y) - \sum_{i=1}^{N_0}  \sum_{(k,\tau)\in M_i} \lambda_{i,k,\tau} \vec b_{i,k,\tau} (y)  \bigg\|_{ \ell^{q'}  } \d \mu (y)  \\
		& \le  \int_B  \bigg\| \sum_{i=N_0+1}^{\infty } \vec f^{(i)} (y)  \bigg\|_{ \ell^{q'}  } \d \mu (y) +  \sum_{i=1}^{N_0} \int_B  \bigg\|  \vec f^{(i)} (y) -   \sum_{(k,\tau)\in M_i} \lambda_{i,k,\tau} \vec b_{i,k,\tau} (y)  \bigg\|_{ \ell^{q'}  } \d \mu (y)  \\
		& \le \| \chi_B \|_{M_{p,\omega}^{t,r} } \bigg( \sum_{i=N_0+1}^{\infty }  \|\vec f^{(i)}\|_{ \H_{p', \omega ' }^{t',r'}( \ell^{q'}  )   }  +   \sum_{i=1}^{N_0}  \bigg\|  \vec f^{(i)}  -   \sum_{(k,\tau)\in M_i} \lambda_{i,k,\tau} \vec b_{i,k,\tau}  \bigg\|_{\H_{p', \omega ' }^{t',r'}( \ell^{q'}  )  }  \bigg) \\
		& \le \| \chi_B \|_{M_{p,\omega}^{t,r} } ( \epsilon +\sum_{i=1}^{N_0}  2^{-i} \epsilon  ) \\
		& \lesssim  \| \chi_B \|_{M_{p,\omega}^{t,r} }  \epsilon.
	\end{align*}
	As a consequence,  $\sum_{i\ge 1} \sum_{k \in \mathbb Z, \tau \in I_k} \lambda_{i,k,\tau} \vec b_{i,k,\tau}$ converges to $\vec f$ in  $L^1_{\operatorname{loc}} ( \ell^{q'}  ) $.  Hence,  $\sum_{i\ge 1} \sum_{k \in \mathbb Z, \tau \in I_k} \lambda_{i,k,\tau} \vec b_{i,k,\tau}$ converges to $\vec f$ locally in measure. Therefore, there exists a subsequence of  
	\begin{equation*}
		\left\{\sum_{i= 0}^{N} \sum_{(k,\tau)\in K \subset Y , \sharp K = M} \lambda_{i,k,\tau} \vec b_{i,k,\tau}  \right\}_{ N \in \mathbb N ,M \in \mathbb N}
	\end{equation*}
	converges to $\vec f$ a.e. 
	Recall that $\sharp K$ is the cardinal number of the set $K$.
	Furthermore, $\lambda_{i,k,\tau}$, $i\in \mathbb N_+, (k,\tau) \in Y$  satisfies
	\begin{equation*}
		\sum_{i\ge 1}	\bigg(    \sum_{ k \in \mathbb Z, \tau \in I_k  }   | \lambda_{i,k,\tau}| ^{r'} \bigg)^{1/r'}  \le \sum_{i\ge 1} (1+\epsilon)  \|    \vec f^{(i)} \|_{ 	\H_{p', \omega ' }^{t',r'}( \ell^{q'}  )  }.
	\end{equation*}
	That is, $ \sum_{i\ge  1}  \vec f^{(i)}$ converges to $\vec f$ in  $	\H_{p', \omega ' }^{t',r'}( \ell^{q'} )$. Since $\epsilon$ is arbitrary, by the definition of block spaces $\H_{p', \omega ' }^{t',r'}( \ell^{q'} )$, we obtain
	\begin{equation*}
		\|\vec f\|_{ \H_{p', \omega ' }^{t',r'}( \ell^{q'} )   }  =  	\bigg\| \sum_{i\ge  1}  \vec f^{(i)} \bigg\|_{ \H_{p', \omega ' }^{t',r'}( \ell^{q'} )   }  \le \sum_{i\ge  1} \|    \vec f^{(i)} \|_{ 	\H_{p', \omega ' }^{t',r'}( \ell^{q'} )  }.
	\end{equation*}
	Then by the well known result that a normed linear space is complete if and only if every absolutely summable sequence is summable (for example, see \cite[Theorem III.3]{RS72}), 	$\H_{p', \omega ' }^{t',r'}( \ell^{q'} ) $ is complete. The proof of scalar valued block space $ \H _{p', \omega '  }^{t',  r' } $ is similar.
\end{proof}

\section{Predual of weighted homogeneous Bourgain-Morrey Besov space} \label{Predual BM Beosv}

\begin{definition} \label{def spa block}
	Let $\mathcal D$ be the dyadic cubes in $X$ as in Remark \ref{dya cube}.
	Let $\psi$ be a partition of unity. Let $ 0< q\le \infty$, $s\in \mathbb R$ and $\omega ' \in A_\infty $. Let $1<p<t<r<\infty$ 	or $1<p\le t<r=\infty$.
	The weighted homogeneous  Besov block space $\dot{B \H}^{ - s,q ',\psi,L}_{p ' ,t ' ,r ' ,\omega ' } (X)$ is defined as follows:
	\[
	\dot{B \H}^{ - s,q ',\psi,L}_{p ' ,t ' ,r ' ,\omega ' } (X) = \{ f\in \mathcal S_\infty ' : \| f \| _{ \dot{B \H}^{ - s,q ',\psi,L}_{p ' ,t ' ,r ' ,\omega ' } (X) } <\infty  \},
	\]
	where
	\[
	\| f\|_{\dot{B \H}^{ - s,q ',\psi,L}_{p ' ,t ' ,r ' ,\omega ' } (X) } =
	\bigg(  \sum_{j\in \mathbb{Z}}  2^{-js q' }\Big\|  \psi_j \big(\sqrt{L}\big)f\Big\|^{q'} _{  \H _{p', \omega'   }^{t',  r' }  }   \bigg)^{1/q'}.
	\]
	The weighted homogeneous 	 Triebel-Lizorkin block space $\dot{F \H}^{-s,q',\psi,L}_{p',t',r',\omega '} (X)$ is defined by
	\[
	\dot{F \H}^{-s,q',\psi,L}_{p',t',r',\omega '} (X) = \{ f\in \mathcal S_\infty ' : \| f \| _{ \dot{F \H}^{-s,q',\psi,L}_{p',t',r',\omega '} (X) } <\infty  \},
	\]
	where
	\[
	\| f\|_{\dot{F \H}^{-s,q',\psi,L}_{p',t',r',\omega '} (X)} =  \bigg\|\bigg( \sum_{j\in \mathbb{Z}} 2^{-jsq'} \Big| \psi_j \big(\sqrt{L}\big)f \Big|^{q'}\bigg)^{1/q'}\bigg\|_{\H _{p', \omega'   }^{t',  r' } }   .
	\]
\end{definition}

We usually need $w\in A_\infty$ for  weighted homogeneous  Besov-Triebel-Lizorkin spaces in \cite{BX25} and $ \omega ' = \omega ^{-1 / (p-1)} \in A_\infty $ for 	the weighted homogeneous  Besov-Triebel-Lizorkin block spaces. By (vi)  of Lemma \ref{weights},  $\omega \in A_p$.

	In what follows, the symbol $\hookrightarrow$ stands for continuous embedding.
We list so-called basic embeddings. 
\begin{lemma}	
		Let $\psi$ be a partition of unity. Let $ 0< q\le \infty$, $s\in \mathbb R$ and $\omega \in A_\infty $.
	Let $1<p<t<r<\infty$ 	or $1<p\le t<r=\infty$.
	
	{\rm (i)} The space $\dot{A \H}^{ - s,q ',\psi,L}_{p ' ,t ' ,r ' ,\omega ' } (X), \mathcal A \in\{ \mathcal B, \mathcal F\}$ is monotone with $q'$, namely, if $q'_1 \le q'_2$, then 
	\begin{equation*}
		\dot{A \H}^{ - s,q '_1 ,\psi,L}_{p ' ,t ' ,r ' ,\omega ' } (X) \hookrightarrow \dot{A \H}^{ - s,q ' _2,\psi,L}_{p ' ,t ' ,r ' ,\omega ' } (X) .
	\end{equation*}
	
	{\rm (ii)} The space $\dot{A \H}^{ - s,q ',\psi,L}_{p ' ,t ' ,r ' ,\omega ' } (X), \mathcal A \in\{ \mathcal B, \mathcal F\}$ is monotone with $r'$, namely, if $r'_1 \le r'_2$, then 
	\begin{equation*}
	\dot{A \H}^{ - s,q ',\psi,L}_{p' ,t',r_1,\omega '} (X) \hookrightarrow\dot{A \H}^{-s,q',\psi,L}_{p',t',r_2,\omega'} (X).
	\end{equation*}
\end{lemma}

\begin{proof}
		The properties  (i) and (ii) are  coming from the monotonicity of the $\ell^q$-norm on $q$.
\end{proof}

\begin{lemma} \label{S_infty embed}
		Let $ 0< q  ' \le  \infty$, $s\in \mathbb R$ and $\omega  \in A_{p}$ where $ p\in (1,\infty)$.
	Let  $\beta_\omega : =  1 - (1- \alpha_0 ) ^p / [\omega ]_{A_p}$ where $\alpha_0$  is the reverse doubling  constant of $X$. 	
Let $ 1< p< t < \min\{ r, \log_2 ( \beta_\omega ^{-1} )  r  /n \}  <\infty$ or let $ 1 <p \le t <r =\infty $.	
Then 	$ \mathcal S_\infty \hookrightarrow  \dot{B \H}^{ - s,q ',\psi,L}_{p ' ,t ' ,r ' ,\omega '} (X) $ and $ \mathcal S_\infty \hookrightarrow  \dot{F \H}^{ - s,q ' ,\psi,L}_{p '  ,t' ,r',\omega'} (X)$.
\end{lemma}
\begin{proof}
	Let $f \in  \mathcal S_\infty$. 	Recall that since $f \in \mathcal S _\infty$, for each $k\in \mathbb N$, there exists $g_k \in \mathcal S$ such that $f = L^k g_k$. For $m>0$ and $\ell , k \in \mathbb N$, we have
	\begin{equation*} 
		\mathcal P _{m, \ell ,k} ^* (f ) = \sup_{x\in X} (1 + \rho (x, x_0)) ^m  \big|  L^\ell g_k (x)    \big|.
	\end{equation*}
First consider the Besov block space
	\begin{align*}
		\| f\|_{\dot{B \H}^{ - s,q ',\psi,L}_{p ' ,t ' ,r ' ,\omega ' } (X) } & =
		\bigg(  \sum_{j\in \mathbb{Z}}  2^{-js q' }\Big\|  \psi_j \big(\sqrt{L}\big)f\Big\|^{q'} _{  \H _{p', \omega'   }^{t',  r' }  }   \bigg)^{1/q'}.
	\end{align*}
Case $j \ge0$. From \cite[p. 903]{BX25},
we have
\begin{equation*}
	|  \psi_j \big(\sqrt{L}\big)f |\lesssim 2^{- j( 2m -n)}  \mathcal P _{m,m,k}^* (f)  ( 1+ \rho (x,x_0) ) ^{-m + \tilde n +n} .
\end{equation*}
Let $  (2m -n > |s|) $ and $ m  -\tilde n - n  >  n + (\log_2 \beta_\omega  )  /t  $ where  $\beta_\omega : =  1 - (1- \alpha_0 ) ^p / [\omega ]_{A_p}$.  By Example \ref{exm ploy}, we have
\begin{equation} \label{j ge 0 S}
	\bigg(  \sum_{j =0}^\infty  2^{-js q' }\Big\|  \psi_j \big(\sqrt{L}\big)f\Big\|^{q'} _{  \H _{p', \omega'   }^{t',  r' }  }   \bigg)^{1/q'} \lesssim  \mathcal P _{m,m,m}^* (f) .
\end{equation}
Case $ j <0$. From \cite[p. 905]{BX25},
\begin{equation*}
	\big| 	\psi_j \big(\sqrt{L}\big) \phi (x) \big| \lesssim  2^{j (m +\tilde n +n) }  \mathcal P _{m,m,m}^* (f) ( 1+ \rho ( x,x_0) )^ {-m +\tilde n +n} .
\end{equation*}
Let $ m +\tilde n +n > |s| $ and $ m -\tilde n - n  > n + (\log_2 \beta_\omega  )  /t $. By Example \ref{exm ploy}, we have
\begin{equation}\label{j < 0 S}
	\bigg(  \sum_{j =-\infty}^{-1} 2^{-js q' }\Big\|  \psi_j \big(\sqrt{L}\big)f\Big\|^{q'} _{  \H _{p', \omega'   }^{t',  r' }  }   \bigg)^{1/q'} \lesssim  \mathcal P _{m,m,m}^* (f) .
\end{equation}
From (\ref{j ge 0 S}), (\ref{j < 0 S}),  we show
\begin{equation*}
	\| f\|_{\dot{B \H}^{ - s,q ',\psi,L}_{p ' ,t ' ,r ' ,\omega ' } (X) } \lesssim \mathcal P _{m,m,m}^* (f) 
\end{equation*}
and  $\mathcal S_\infty \hookrightarrow \dot{B \H}^{ - s,q ',\psi,L}_{p ' ,t ' ,r ' ,\omega '} (X)$.

The proof of  $\mathcal S_\infty \hookrightarrow  \dot{F \H}^{ - s,q ',\psi,L}_{p ' ,t ' ,r ' ,\omega '} (X) $  is similar and we omit it here.
\end{proof}

\begin{proposition} \label{pro ell H = ell BM}
	 Let $\omega \in A_p$ for $p \in (1,\infty)$. Let $\beta_\omega : =  1 - (1- \alpha_0 ) ^{ p }  / [\omega ]_{A_  {p} }$ where $\alpha_0$  is the reverse doubling  constant of $X$. 
	Let $ 1< p< t < \min\{ r, \log_2 ( \beta_\omega ^{-1} )  r  /n \}  <\infty$ or let $ 1 <p \le t <r =\infty $.
	Let 
	\begin{equation*}
		q = \begin{cases}
			\frac{q'}{q' -1}, & {\rm if} \;  q' \in (1,\infty) , \\
			\infty, & {\rm if} \;  q' \in (0,1] .
		\end{cases} 
	\end{equation*}	
	 Then
	\begin{equation*}
	 \left( \ell^{q'} ( \H _{p', \omega '  }^{t',  r' } )  \right)^\ast =  \ell^q ( M_{p,\omega}^{t,r} ) .
	\end{equation*}
That is $\vec  g \in \left( \ell^{q'} ( \H _{p', \omega '  }^{t',  r' } )  \right)^\ast $ if and only if it can be represented uniquely as 
\begin{equation*}
\vec  g(\vec  f)  = \sum_{j \in \mathbb Z} \int_X g_j (x) f_j (x) \d \mu (x)
\end{equation*}
for every $\vec  f \in  \ell^{q'} ( \H _{p', \omega '  }^{t',  r' } ) $ where $  \{ g_j\}_{j \in \mathbb Z} \in   \ell^q ( M_{p,\omega}^{t,r} )$  and $ \| \vec  g\|_{ \left( \ell^{q'} ( \H _{p', \omega '  }^{t',  r' } )  \right)^\ast }  = \| \{ g_j\}  \|_{\ell^q ( M_{p,\omega}^{t,r} ) }$.

\end{proposition}

\begin{proof}
	Step 1. By Theorem \ref{predual scalar}, 
	\begin{align*}
		|\vec  g( \vec f) |  & \le  \sum_{j \in \mathbb Z} \int_X  |g_j (x) f_j (x)| \d \mu (x) \\
		& \le \sum_{j \in \mathbb Z}  \| g_j \|_{M_{p,\omega}^{t,r}  }  \| f_j \|_{ \H _{p', \omega '  }^{t',  r' } }  \\
		& \le  \| \{ g_j\}  \|_{\ell^q ( M_{p,\omega}^{t,r} ) }   \| \{ f_j\}  \|_{ \ell^{q'} ( \H _{p', \omega '  }^{t',  r' } ) } ,
	\end{align*}
where we used H\"older's inequality when $q ' \in (1,\infty) $  and $\ell^{q'} \hookrightarrow \ell^1 $ when $q' \in (0,1] $.
Taking the supremum over $  \| \{ f_j\}  \|_{ \ell^{q'} ( \H _{p', \omega '  }^{t',  r' } ) } \le 1$, we obtain
\begin{equation*}
	\|\vec  g\|_{ \left( \ell^{q'} ( \H _{p', \omega '  }^{t',  r' } )  \right)^\ast }  \le  \| \{ g_j\}  \|_{\ell^q ( M_{p,\omega}^{t,r} ) }.
\end{equation*}
Step 2. Let $ g \in \left( \ell^{q'} ( \H _{p', \omega '  }^{t',  r' } )  \right)^\ast $. 
Put 
\begin{equation*}
	g_j (\vec f) = \vec g (  \{  0, \ldots, 0, f_j, 0, \ldots\}     ), \quad  \vec f= \{f_j \}_{j \in \mathbb Z} \in \ell^{q'} ( \H _{p', \omega '  }^{t',  r' } ).
\end{equation*}
Then $g_j $ belongs to $ ( \H _{p', \omega '  }^{t',  r' } ) ^\ast $.
By the representation of elements of $ ( \H _{p', \omega '  }^{t',  r' } ) ^\ast $ (Theorem \ref{predual scalar}), there exists a unique function $  g_j \in M_{p,\omega}^{t,r}$ such that 
\begin{equation*}
		g_j (\vec f) = \int_X   g_j (x) f_j (x) \d \mu (x), \quad \|g_j \|_{ ( \H _{p', \omega '  }^{t',  r' } ) ^\ast  }  = \|  g_j \|_{  M_{p,\omega}^{t,r} } , 
\end{equation*}
where $g_j$ in the last formula is interpreted as an element of $ ( \H _{p', \omega '  }^{t',  r' } ) ^\ast $. 
It follows that 
\begin{equation} \label{step 2 key ineq}
\sum_{|j| \le N } 	\int_X   g_j (x) f_j (x) \d \mu (x) \le \|\vec g\|_{ \left( \ell^{q'} ( \H _{p', \omega '  }^{t',  r' } )  \right)^\ast   }  \| \{f_j\}_{ |j|\le N }  \|_{ \ell^{q'} ( \H _{p', \omega '  }^{t',  r' } )  } .
\end{equation}
Since $g_j \in M_{p,\omega}^{t,r} $,  for each $k\in \mathbb Z $, $\tau \in I_k$, and $x\in X$, let
\begin{equation} \label{def b_k,tau j}
	b_{k,\tau}^j  := \begin{cases} \omega (Q^k_\tau ) ^{1/t-1/p} \|  g_j\|_{ L^p (\omega, Q^k_\tau) } ^{1-p}   |g_j(x)|^{p-1} \omega (x)   \chi_{ Q^k_\tau } (x)  \sgn  g_j (x)
		& {\rm if}\; \|  g_j \|_{ L^p (\omega, Q^k_\tau) } \neq 0, \\
		0, & {\rm if}\; \| g_j  \|_{ L^p (\omega, Q^k_\tau) }= 0.
	\end{cases}
\end{equation} 
Then for $ \|  g_j \|_{ L^p (\omega, Q^k_\tau) } \neq 0 $,
\begin{align*}
	\| b_{k,\tau} ^j  \|_{L^{p'} (\omega ' , Q^k_\tau)  } 
	& = \omega (Q^k_\tau ) ^{1/t-1/p}  .
\end{align*}
Hence $b_{k,\tau} ^j$ is a $ (p ', t', \omega ')$-block supported in $Q^k_\tau $. 
Now for $ \{ \lambda_{k,\tau }\}_{k \in \mathbb Z, \tau \in I_k  }  \in \ell^{r'}$, let 
\begin{equation} \label{f_j by tilde g}
	f_j = \sum_{ k \in \mathbb Z, \tau \in I_k  }   \lambda_{k,\tau }^j  b_{k,\tau} ^j   \in  \H _{p', \omega '  }^{t',  r' } .
\end{equation}
Put (\ref{f_j by tilde g}) into (\ref{step 2 key ineq}), and we obtain
\begin{align} \label{new estimate sum tilde g_j f_j}
	\nonumber
\sum_{|j| \le N } 	\int_X    g_j (x) f_j (x) \d \mu (x) & = \sum_{|j| \le N } \sum_{ k \in \mathbb Z, \tau \in I_k  }   \lambda_{k,\tau } ^j  \omega (Q^k_\tau ) ^{1/t-1/p}  \left(  \int_{ Q^k _\tau } |  g_j (x) |^p \omega (x) \d \mu (x) \right) ^{1/p}  \\
\nonumber
& \le  \|\vec  g\|_{ \left( \ell^{q'} ( \H _{p', \omega '  }^{t',  r' } )  \right)^\ast   }  \| \{f_j\}_{ |j|\le N }  \|_{ \ell^{q'} ( \H _{p', \omega '  }^{t',  r' } )   }  \\
& \le \|\vec g\|_{ \left( \ell^{q'} ( \H _{p', \omega '  }^{t',  r' } )  \right)^\ast   }  \left( \sum_{ |j|\le N }   \left(  \sum_{ k \in \mathbb Z, \tau \in I_k  }   |\lambda_{k,\tau } ^j |^{r'}  \right)^{q'/r'}   \right)^{1/q'}.
\end{align}
Case $ q' \in (1,\infty) $.
Then  by (\ref{new estimate sum tilde g_j f_j}) and  the duality of mixed-norm Lebesgue sequence space (\cite[Theorem 2]{BP61}),
\begin{align*}
	\| \vec g \|_{ \ell^{ q} (M_{p,\omega}^{t,r}  )  } & = \left(  \sum_{j \in \mathbb Z} \|  g_j \|_{ M_{p,\omega}^{t,r} }    ^q  \right)^{1/q} \\
	 &  =  \left(  \sum_{j \in \mathbb Z}  \left(\sum_{ k \in \mathbb Z, \tau \in I_k  }   \omega (Q^k_\tau ) ^{r/t-r/p}  \left(  \int_{ Q^k _\tau } | g_j (x) |^p \omega (x) \d \mu (x) \right) ^{r/p}   \right)     ^q  \right)^{1/q} \\
	&= \sup_{ \| \{  \lambda_{k,\tau }^j\} \|_{\ell^{q'}_j (   \ell^{r' }_{k,\tau}  )  }  \le 1 }  \sum_{j \in \mathbb Z}   \sum_{ k \in \mathbb Z, \tau \in I_k  }   \lambda_{k,\tau }^j  \omega (Q^k_\tau ) ^{1/t-1/p}  \left(  \int_{ Q^k _\tau } | g_j (x) |^p \omega (x) \d \mu (x) \right) ^{1/p}  
	 \\
	& \le \|\vec g\|_{ \left( \ell^{q'} ( \H _{p', \omega '  }^{t',  r' } )  \right)^\ast   }  .
\end{align*}
Case $ q' \in (0,1] $. 
First note that \begin{align} \label{g_j norm le g norm}
	\nonumber
	\| g_j \|_{ ( \H _{p', \omega '  }^{t',  r' } ) ^\ast  }  & =  \sup_{ \|\vec f\|_{  \ell^{q'} ( \H _{p', \omega '  }^{t',  r' } )   }  \le 1   }    | \vec g (  \{  0, \ldots, 0, f_j, 0, \ldots\}     ) | \\
	\nonumber
	& \le \sup_{ \|\vec f\|_{  \ell^{q'} ( \H _{p', \omega '  }^{t',  r' } )   }  \le 1   }    \| \vec g\|_{  \left( \ell^{q'} ( \H _{p', \omega '  }^{t',  r' } )  \right)^\ast   }  \| \{  0, \ldots, 0, f_j, 0, \ldots\} \|_{ \ell^{q'} ( \H _{p', \omega '  }^{t',  r' } )   }  \\
	& \le \| \vec g\|_{  \left( \ell^{q'} ( \H _{p', \omega '  }^{t',  r' } )  \right)^\ast   } .
\end{align}
Taking the supremum over $j \in \mathbb Z$, we obtain 
\begin{equation*}
	\sup_{ j \in \mathbb Z}\| g_j \|_{ ( \H _{p', \omega '  }^{t',  r' } ) ^\ast  } \le  \| \vec g\|_{  \left( \ell^{q'} ( \H _{p', \omega '  }^{t',  r' } )  \right)^\ast   }  .
\end{equation*}
By (\ref{def b_k,tau j}) and (\ref{g_j norm le g norm}),
\begin{align*}
	\| \vec g \|_{ \ell^{\infty  } (M_{p,\omega}^{t,r}  )  } 
	& =  \sup_{ j \in \mathbb Z} \sup_{ \|  \{  \lambda_{k,\tau }^j \} \|_{ \ell^{r' }_{k,\tau}  }  \le 1  }  \sum_{ k \in \mathbb Z, \tau \in I_k  }   \lambda_{k,\tau }^j  
	\int_{  Q^k _\tau } | b_{k,\tau}^j (x) | |g_j(x) | \d \mu  (x)
   \\
	& = \sup_{ j \in \mathbb Z} \sup_{ \|  \{  \lambda_{k,\tau }^j \} \|_{ \ell^{r' }_{k,\tau}  }  \le 1  }
	\int_X |g_j (x) | \left(  \sum_{ k \in \mathbb Z, \tau \in I_k  }   \lambda_{k,\tau }^j | b_{k,\tau}^j (x) |  \right) \d \mu  (x) \\	
	& \le  \sup_{ j \in \mathbb Z} \sup_{ \|  \{  \lambda_{k,\tau }^j \} \|_{ \ell^{r' }_{k,\tau}  }  \le 1  }
	\| g_j \|_{ ( \H _{p', \omega '  }^{t',  r' } ) ^\ast  }  \left\|   \sum_{ k \in \mathbb Z, \tau \in I_k  }   \lambda_{k,\tau }^j  |b_{k,\tau}^j|      \right\|_{  \H _{p', \omega '  }^{t',  r' }  } \\
	&\le \sup_{ j \in \mathbb Z}\| g_j \|_{ ( \H _{p', \omega '  }^{t',  r' } ) ^\ast  } \le \| \vec g\|_{  \left( \ell^{q'} ( \H _{p', \omega '  }^{t',  r' } )  \right)^\ast   }  .
\end{align*}
Thus we complete the proof.
\end{proof}
Next we obtain the main result of this section.

\begin{theorem} \label{predual Besov}
	Let   $s\in \mathbb R$ and $\omega \in A_p$ for $ p \in (1,\infty)$.	
	Let $\beta_\omega : =  1 - (1- \alpha_0 ) ^{ p }  / [\omega ]_{A_  {p} }$ where $\alpha_0$  is the reverse doubling  constant of $X$. 
	Let $ 1< p< t < \min\{ r, \log_2 ( \beta_\omega ^{-1} )  r  /n \}  <\infty$ or let $ 1 <p \le t <r =\infty $.
		Let 
	\begin{equation*}
		q = \begin{cases}
			\frac{q'}{q' -1}, & {\rm if} \;  q' \in (1,\infty) , \\
			\infty, & {\rm if} \;  q' \in (0,1] .
		\end{cases} 
	\end{equation*}	
	Then
	\begin{align*}
	\left( \dot{B \H}^{ - s,q' ,\psi,L}_{p ',t' ,r ',\omega ' } (X) \right) ^\ast =	\dot{\mathcal B}^{s,q,\psi,L}_{p,t,r,\omega} (X) .
	\end{align*}	
\end{theorem}

\begin{proof}
	First we prove the inclusion $\dot{\mathcal B}^{s,q,\psi,L}_{p,t,r,\omega} (X)   \subset 	\left( \dot{B \H}^{ - s,q' ,\psi,L}_{p ',t' ,r ',\omega ' } (X) \right) ^\ast $. Let $ f \in \dot{\mathcal B}^{s,q,\psi,L}_{p,t,r,\omega} (X) $. Our goal is to prove that $f \in 	\left( \dot{B \H}^{ - s,q' ,\psi,L}_{p ',t' ,r ',\omega ' } (X)  \right) ^\ast$ and 
	\begin{equation} \label{goal 1}
		\| f\|_{ 	\left(\dot{B \H}^{ - s,q' ,\psi,L}_{p ',t' ,r ',\omega ' } (X)  \right) ^\ast } \le c \| f\|_{ \dot{\mathcal B}^{s,q,\psi,L}_{p,t,r,\omega} (X) } .
	\end{equation}
	Let $\psi \in C_0^\infty (\mathbb R _+) $ be a real-valued function  such that supp $\psi  \subset [1/2,2]$ and
\begin{equation*}
	\sum_{j\in \mathbb Z} \psi^2 ( 2^{-j} \lambda ) =1
\end{equation*}
for all $ \lambda \in \mathbb R _+$. Set $\psi_j (\lambda) = \psi(2^ {-j} \lambda), j\in \mathbb Z$. Then $\sum_{j \in \mathbb Z}  \psi_{j}^2 (\lambda) =1 $ for $\lambda \in \mathbb R_+$. Hence, by Lemma \ref{conv in S infty}, we have
\begin{equation} \label{eq f = sum psi 2}
	f = \sum_{j\in \mathbb Z} \psi_j ^2 \big(\sqrt{L}\big) f  \quad \operatorname{in} \;  \mathcal S _\infty '.
\end{equation}	
By (\ref{eq f = sum psi 2}), 	for every $\phi \in \mathcal S_\infty $, we have
	\begin{equation*}
		\langle f, \phi\rangle = \sum_{j= \mathbb Z } 	\langle \psi_j (\sqrt{L})  f, \psi_j (\sqrt{L}) \phi\rangle .
	\end{equation*}
Then by Theorem \ref{predual scalar},
\begin{align*}
	|	\langle f, \phi\rangle | & \le \sum_{j= \mathbb Z } | \langle  \psi_j (\sqrt{L})  f ,  \psi_j (\sqrt{L}) \phi  \rangle | \\
	& \le \sum_{j= \mathbb Z }  \left\|   2^{js} \psi_j (\sqrt{L})  f  \right\|_{ M_{p, \omega }^{t,r}   } \left\|  2^{-js} \psi_j (\sqrt{L}) \phi   \right\|_{ \H _{p', \omega '  }^{t',  r' } } .
\end{align*}
For $1< q' <\infty$, apply the H\"older inequality and get
\begin{equation} \label{q ge 1}
	|	\langle f, \phi\rangle | \le   \left(  \sum_{j= \mathbb Z }  \left\|   2^{js} \psi_j (\sqrt{L})  f  \right\|_{ M_{p, \omega }^{t,r}   } ^{q}  \right) ^{1/q}   \left(  \sum_{j= \mathbb Z }  \left\|  2^{-js} \psi_j (\sqrt{L}) \phi   \right\|_{ \H _{p', \omega '  }^{t',  r' }  } ^{q'}  \right) ^{1/q'} .
\end{equation}
For $ 0<q' \le 1$, using $ \ell^{q'} \hookrightarrow \ell^1$ we have
\begin{align} \label{Holder q' < 1}
	\nonumber
		|	\langle f, \phi\rangle | & \le \sup_{j\in \mathbb Z}  \left\|  2^{js} \psi_j (\sqrt{L}) \phi   \right\|_{ M_{p, \omega }^{t,r} }  \sum_{j= \mathbb Z }  \left\|   2^{-js} \psi_j (\sqrt{L})  f  \right\|_{ \H _{p', \omega '  }^{t',  r' }  }  \\
		& \le  \sup_{j\in \mathbb Z}  \left\|  2^{js} \psi_j (\sqrt{L}) \phi   \right\|_{ M_{p, \omega }^{t,r} }  \left(  \sum_{j= \mathbb Z }  \left\|   2^{-js} \psi_j (\sqrt{L})  f  \right\|_{ \H _{p', \omega '  }^{t',  r' }  }  ^{q'} \right) ^{ 1/q' }  .
\end{align}
From (\ref{q ge 1}) and (\ref{Holder q' < 1}), we obtain (\ref{goal 1}).

Next we proceed with the embedding $\left(\dot{B \H}^{ - s,q' ,\psi,L}_{p ',t' ,r ',\omega ' } (X)  \right) ^\ast  \subset \dot{\mathcal B}^{s,q,\psi,L}_{p,t,r,\omega} (X)  $. 

Case $ q' \in [1,\infty )  $. In this case,  $ \ell^{q'} ( \H _{p', \omega '  }^{t',  r' } ) $ is a Banach space by Theorem \ref{block space is Banach}.
Let $Y$ be the subspace of the normed space $ \ell^{q'} ( \H _{p', \omega '  }^{t',  r' } ) $ defined by
\begin{equation*}
	Y:= \left\{     \{  2^{ - js} \psi_j (\sqrt{L}) g  \} _{j \in \mathbb Z } : g \in  \dot{B \H}^{ - s,q' ,\psi,L}_{p ',t' ,r ',\omega ' } (X)   \right\} .
\end{equation*}
Note that the mapping $\rho:\dot{B \H}^{ - s,q' ,\psi,L}_{p ',t' ,r ',\omega ' } (X)   \to Y,$  $g  \mapsto \{  2^{ - js} \psi_j (\sqrt{L}) g  \} _{j \in \mathbb Z } $ is an isometric isomorphism.
If $f \in\left( \dot{B \H}^{ - s,q' ,\psi,L}_{p ',t' ,r ',\omega ' } (X)   \right) ^\ast $, we define $f_0 :=  f \circ  \rho^{-1} :  Y \to \mathbb C $ and note that $f_0 \in Y ^\ast  $ with $ \|f_0 \|_{ Y^\ast}  = \| f\|_{ \left( \dot{B \H}^{ - s,q' ,\psi,L}_{p ',t' ,r ',\omega ' } (X) \right) ^\ast }$. Then according to the Hahn-Banach Theorem there exists an
extension $\tilde  f  \in   \left(  \ell^{q'} ( \H _{p', \omega '  }^{t',  r' } ) \right) ^\ast $  of $f_0$ preserving the operator norm i.e., 
\begin{equation*}
	\| \tilde f \|_{  \left(  \ell^{q'} ( \H _{p', \omega '  }^{t',  r' } ) \right) ^\ast  } = \|f_0 \|_{Y^\ast }  = \| f\|_{  \left( \dot{B \H}^{ - s,q' ,\psi,L}_{p ',t' ,r ',\omega ' } (X)  \right) ^\ast  }.
\end{equation*}
Moreover, by Proposition \ref{pro ell H = ell BM}, there exists $  \{  f_j \}_{j\in \mathbb Z} \in \ell^q ( M_{p,\omega}^{t,r} ) $ representing $\tilde f$, that is 
\begin{equation*}
\tilde f (\vec  g) = \sum_{j  \in \mathbb Z}  \int_X f_j (x) g_j (x) \d \mu (x)
\end{equation*}
for each $\vec g  = \{  g_j \}_{j\in \mathbb Z} \in \ell^{q'} ( \H _{p', \omega '  }^{t',  r' } )  $ and $ \| \tilde f \| _{ \left(  \ell^{q'} ( \H _{p', \omega '  }^{t',  r' } ) \right) ^\ast  } = \|  \{  f_j \}_{j\in \mathbb Z} \|_{ \ell^q ( M_{p,\omega}^{t,r} )  }  $.

Now let $\phi \in \mathcal S_\infty$, since $\mathcal S_\infty \subset \mathcal S  \subset  \dot{B \H}^{ - s,q' ,\psi,L}_{p ',t' ,r ',\omega ' } (X)  $ (Lemma \ref{S_infty embed}), by definition $\rho (\phi)  =\{  2^{ - js} \psi_j (\sqrt{L}) \phi \} _{j \in \mathbb Z }   \in Y $. Then 
\begin{align*}
	\langle f, \phi \rangle & = f(\bar \phi) =  (f \circ \rho^{-1} ) ( \rho (\bar \phi) ) = \tilde{f} ( \rho (\bar \phi) ) \\
	&  = \sum_{j \in \mathbb Z} \int_X f_j (x) 2^{ -js} \psi_j (\sqrt{L}) \bar \phi (x) \d \mu (x) \\
	&= \sum_{j \in \mathbb Z} \int_X f_j (x) 2^{-js} \int_X K_{\psi_j (\sqrt L)} (x,y)  \bar \phi (y) \d \mu (y) \d \mu (x) \\
	& = \sum_{j \in \mathbb Z} \int_X \int_X K_{\psi_j (\sqrt L)} (x,y)f_j (x) 2^{ -js} \d \mu (x)\bar \phi (y) \d \mu (y)  \\
	& = \sum_{j \in \mathbb Z} \int_X  \psi_j (\sqrt L )  (2^{ -js} f_j  ) (y)  \bar \phi (y) \d \mu (y)  \\
	& = \sum_{j \in \mathbb Z} \langle \psi_j (\sqrt L )  (2^{ -js} f_j  ), \phi \rangle ,
\end{align*}
which implies that 
\begin{equation} \label{f = sum_j in S _infty dual}
	f = \sum_{j \in \mathbb Z}  \psi_j (\sqrt L )  (2^{-js} f_j  )
\end{equation}
in the sense of $\mathcal S_\infty' $.
To justify the above change in the order of integration, by Lemma \ref{kernel est} and Theorem \ref{predual scalar}, we have
\begin{align*}
 &	\int_X \int_X  | K_{\psi_j (\sqrt L)} (x,y)f_j (x) 2^{ -js}   \phi (y) | \d \mu (x)  \d \mu (y)  \\
 & \lesssim 2^{-js} \int_X \int_X   \frac{1}{ |B(y, 2^{-j})   |  (1+ 2^j \rho (x,y)) ^{-\sigma}  } |  f_j (x)    \phi (y) |\d \mu (x) \d \mu (y)  \\
 & \lesssim \| f_j \|_{ M_{p,\omega}^{t,r} } 2^{-js}  \left\| \int_X \frac{1}{ |B(y, 2^{-j})   |  (1+ 2^j \rho (x,y)) ^{-\sigma}  } | \phi (y) | \d \mu (y)  \right\|_{\H _{p', \omega '  }^{t',  r' } }  .
\end{align*}
Recall that since $\phi \in \mathcal S _\infty$, for each $k\in \mathbb N$, there exists $\phi_k \in \mathcal S$ such that $\phi = L^k \phi_k$. For $m>0$ and $\ell , k \in \mathbb N$, we have
\begin{equation*} 
	\mathcal P _{m, \ell ,k} ^* (\phi) = \sup_{x\in X} (1 + \rho (x, x_0)) ^m  \big|  L^\ell \phi_k (x)    \big|.
\end{equation*}
Choosing  $\sigma >n$ and  $m > n + (\log_2 \beta_\omega  )  /t$, by Lemma \ref{basic est}, Example \ref{exm ploy} and Theorem \ref{predual scalar},  we obtain
\begin{align*}
&	\left\| \int_X \frac{1}{ |B(y, 2^{-j})   |  (1+ 2^j \rho (x,y)) ^{-\sigma}  } | \phi (y) | \d \mu (y)  \right\|_{\H _{p', \omega '  }^{t',  r' } }  \\
& = \sup_{\|h\|_ { M_{p,\omega}^{t,r} }\le 1 } \int_X \int_X \frac{1}{ |B(y, 2^{-j})   |  (1+ 2^j \rho (x,y)) ^{-\sigma}  } | \phi (y) | \d \mu (y) h (x) \d \mu (x) \\
& \le  \sup_{\|h\|_ { M_{p,\omega}^{t,r} }\le 1 } 	\mathcal P _{m, 0 ,0} ^* (\phi)   \int_X \int_X \frac{1}{ |B(y, 2^{-j})   |  (1+ 2^j \rho (x,y)) ^{-\sigma}  }   h (x)  \d \mu (x) ( 1+ \rho (y,x_0))^{-m} \d \mu (y) \\
& \lesssim  \sup_{\|h\|_ { M_{p,\omega}^{t,r} }\le 1 } \mathcal P _{m, 0 ,0} ^* (\phi)   \int_X \M h (y) ( 1+ \rho (y,x_0))^{-m} \d \mu (y) \\
& \le \sup_{\|h\|_ { M_{p,\omega}^{t,r} }\le 1 } \mathcal P _{m, 0 ,0} ^* (\phi)  \| \M h \|_{ M_{p,\omega}^{t,r} }  \| ( 1+ \rho ( \cdot,x_0))^{-m} \|_{\H _{p', \omega '  }^{t',  r' }  } \\
& \lesssim  \mathcal P _{m, 0 ,0} ^* (\phi) .
\end{align*}
Now getting back into (\ref{f = sum_j in S _infty dual}).
Next we will show $f \in \dot{\mathcal B}^{s,q,\psi,L}_{p,t,r,\omega} (X)$ and that 
\begin{equation*}
	\| f \|_{\dot{\mathcal B}^{s,q,\psi,L}_{p,t,r,\omega} (X) } \lesssim \| f\|_{  \left( \dot{B \H}^{ - s,q' ,\psi,L}_{p ',t' ,r ',\omega ' } (X)   \right) ^\ast}.
\end{equation*}
Let $\ell \in \mathbb Z$. Then 
\begin{align*}
	\psi_\ell (\sqrt{ L}) f (x)& =  \sum_{  j = \ell -1 } ^{\ell  +1} 	\psi_\ell (\sqrt{ L})  \psi_j  (\sqrt{ L}) (2^{-js}  f_j  ) (x) .
\end{align*}
It follows that 
\begin{equation} \label{f dot B le sum}
	 \| f \|_{ \dot{\mathcal B}^{s,q,\psi,L}_{p,t,r,\omega} (X) } \lesssim \left(\sum_{\ell \in \mathbb Z}  2^{\ell s q } \left( \sum_{  j = \ell -1 } ^{\ell  +1}  \| 	\psi_\ell (\sqrt{ L})  \psi_j  (\sqrt{ L}) (2^{ -js } f_j  ) \|_{ M_{p,\omega}^{t,r} }  \right)  ^{q}    \right)  ^{1/q} .
\end{equation}
Since $  |j-\ell | \le 1  $, by Lemma \ref{kernel est},
\begin{align*}
	| K_{	\psi_\ell (\sqrt{ L})}  (x, y ) |  & \lesssim D_{2^{ -\ell }, \sigma } (x,y) \lesssim D_{2^{ -j }, \sigma } (x,y) ,\\
	| K_{	\psi_j (\sqrt{ L})}  (y, z ) |  & \lesssim D_{2^{ -j }, \sigma } (y, z ) ,
\end{align*}
where $\sigma$  is arbitrarily large at our disposal. 
Here and in what follows,
\begin{equation*}
	D_{ \delta , \sigma } (x,y) =  ( V (x, \delta)     V (y, \delta )  ) ^{-1/2}   ( 1+  \delta ^{-1}  \rho (x,y)  ) ^{-\sigma} , \quad  x,y \in X, \delta , \sigma >0.
\end{equation*} 
 Moreover, 
\begin{equation*}
	K_{ 	\psi_\ell (\sqrt{ L})  \psi_j  (\sqrt{ L}) } (x,z) = \int_X K_{	\psi_\ell (\sqrt{ L})   } (x,y) K_{	 \psi_j  (\sqrt{ L})}  (y,z) \d \mu (y).
\end{equation*}
Assume that $  \sigma > 2 n +|s|$, we have
\begin{equation} \label{kernel estimate psi ell psi j}
|	K_{ 	\psi_\ell (\sqrt{ L})  \psi_j  (\sqrt{ L}) } (x,z) | \lesssim  D_{2^{ -j }, \sigma } (x, z ) .
\end{equation}
Then by Lemma \ref{basic est}
\begin{align} \label{psi ell psi j control by M}
	| 	\psi_\ell (\sqrt{ L})  \psi_j  (\sqrt{ L}) (  2^{-js} f_j)  (x)  | &  \lesssim \int_X D_{2^{ -j }, \sigma } (x, y ) | 2^{-js} f_j (y) \d \mu (y) \lesssim \M (2^{-js} f_j  ) (x).
\end{align}
Hence 
\begin{align*}
	\| 	\psi_\ell (\sqrt{ L})  \psi_j  (\sqrt{ L})  (  2^{-js} f_j) \|_{ M_{p,\omega}^{t,r} } \lesssim \| 2^{-js} f_j\|_{M_{p,\omega}^{t,r} }.
\end{align*}
Finally using this last inequality in (\ref{f dot B le sum}), we have
\begin{align*}
	\| f\|_{  \dot{B}^{ s, q ,\psi,L}_{p  , t  ,r ,\omega} (X) } &  \lesssim \left(\sum_{\ell \in \mathbb Z}  2^{\ell s q } \left( \sum_{  j = \ell -1 } ^{\ell  +1}  \| 2^{-js} f_j\|_{M_{p,\omega}^{t,r} }  \right)  ^{q}    \right)  ^{1/q} \\
	&\lesssim  \| \{f_j \}_{j\in \mathbb Z}   \|_{ \ell^{q} (M_{p,\omega}^{t,r} )  } \lesssim   \|f\|_{ \left( \dot{B \H}^{ - s,q' ,\psi,L}_{p ',t' ,r ',\omega ' } (X)   \right) ^\ast }.
\end{align*}

Case $q ' \in (0,1)$.
Since in this case $\ell^{q'} ( \H _{p', \omega '  }^{t',  r' } ) $ is not a normed space, the Hahn-Banach Theorem is not available.
Let $ f \in \left( \dot{B \H}^{ - s,q' ,\psi,L}_{p ',t' ,r ',\omega ' } (X)   \right) ^\ast$.  To show $f \in \dot{B}^{ s, \infty ,\psi,L}_{p  , t  ,r ,\omega} (X) $, we shall prove that 
\begin{equation*}
	\sup_{j\in \mathbb Z} 2^{js } \| \psi_j (\sqrt{L}) f  \|_{M_{p,\omega}^{t,r} } \lesssim \|f\|_{ \left( \dot{B \H}^{ - s,q' ,\psi,L}_{p ',t' ,r ',\omega ' } (X)   \right) ^\ast }.
\end{equation*}
To this end, using Theorem \ref{predual scalar},
it suffices to show 
\begin{equation} \label{q' < 1 Besov aim}
\left| 	\int_X 2^{js}  \psi_j (\sqrt{L} ) f  (x) h (x)  \d \mu  (x)  \right|  \lesssim \|f\|_{ \left( \dot{B \H}^{ - s,q' ,\psi,L}_{p ',t' ,r ',\omega ' } (X)   \right) ^\ast } 
\end{equation}
for all $ j \in \mathbb Z$ and $h \in  \H _{p', \omega '  }^{t',  r' }$ with $\| h\|_{ \H _{p', \omega '  }^{t',  r' } }  \le 1 $.
Since $f \in\left( \dot{B \H}^{ - s,q' ,\psi,L}_{p ',t' ,r ',\omega ' } (X)   \right) ^\ast $,
\begin{align*}
	 \left| 	\int_X 2^{js}  \psi_j (\sqrt{L} ) f  (x) h (x)  \d \mu  (x)  \right|  & =  \left|  \int_X  f (y)    2^{js}  \psi_j (\sqrt{L} )   h (y)    \d \mu (y)  \right| \\
	 & \le \| f\|_{ \left( \dot{B \H}^{ - s,q' ,\psi,L}_{p ',t' ,r ',\omega ' } (X)   \right) ^\ast }  \left\|  2^{js}  \psi_j (\sqrt{L} )   h   \right\|_{\dot{B \H}^{ - s,q' ,\psi,L}_{p ',t' ,r ',\omega ' } (X)  } .
\end{align*}
Now we estimate $ \left\|  2^{js}  \psi_j (\sqrt{L} )   h   \right\|_{\dot{B \H}^{ - s,q' ,\psi,L}_{p ',t' ,r ',\omega ' } (X)  } $.
\begin{align*}
	\| 2^{js } \psi_j (\sqrt{L})h\|_{\dot{B \H}^{ - s,q' ,\psi,L}_{p ',t' ,r ',\omega ' } (X)  } & = \left(  \sum_{\ell \in \mathbb Z}   2^{-\ell s q' } \| 2^{js } \psi_\ell (\sqrt{L}) \psi_j (\sqrt{L}) (h) \|_{  \H _{p', \omega '  }^{t',  r' }   }  ^{q'}   \right) ^{1/q'}  \\
	& = \left(  \sum_{\ell= j-1}^{j+1}   2^{-\ell s q' } \| 2^{js } \psi_\ell (\sqrt{L}) \psi_j (\sqrt{L}) (h) \|_{  \H _{p', \omega '  }^{t',  r' }   }  ^{q'}   \right) ^{1/q'} .
\end{align*}
Using kernel estimate (\ref{kernel estimate psi ell psi j}), Theorem \ref{predual scalar} and Lemma \ref{M bour weight},
\begin{align*}
\|  \psi_\ell (\sqrt{L}) \psi_j (\sqrt{L}) (h) \|_{  \H _{p', \omega '  }^{t',  r' }   }  &=  \sup_{  \| E\|_{ M_{p,\omega}^{t,r} } \le 1} \int_X \psi_\ell (\sqrt{L}) \psi_j (\sqrt{L}) (h) (x) E(x) \d \mu (x) \\
& \lesssim \sup_{  \| E\|_{ M_{p,\omega}^{t,r} } \le 1} \int_X \int_X D_{2^{ -j }, \sigma } (x, y ) h (y) \d \mu  (y) E (x) \d \mu (x) \\
& \lesssim \sup_{  \| E\|_{ M_{p,\omega}^{t,r} } \le 1} \int_X  h (y) \M E (y) \d \mu (y) \\
& \le \sup_{  \| E\|_{ M_{p,\omega}^{t,r} } \le 1}  \| h\|_{ \H _{p', \omega '  }^{t',  r' }   }  \| \M E\|_{M_{p,\omega}^{t,r}  } \\
&\lesssim  \| h\|_{ \H _{p', \omega '  }^{t',  r' }   } \le 1.
\end{align*}
Thus, we obtain $\| 2^{js } \psi_j (\sqrt{L})h\|_{\dot{B \H}^{ - s,q' ,\psi,L}_{p ',t' ,r ',\omega ' } (X)  } \lesssim  1$ and (\ref{q' < 1 Besov aim}). Hence the proof is complete.
\end{proof}

\section{Predual of weighted homogeneous Bourgain-Morrey Triebel-Lizorkin space} \label{Preudal BM TL}
\begin{theorem}\label{predual TL}
	Let   $s\in \mathbb R$ and $\omega \in A_p$ for $p \in (1,\infty)$.	
	Let $\beta_\omega : =  1 - (1- \alpha_0 ) ^{ p }  / [\omega ]_{A_  {p} }$ where $\alpha_0$  is the reverse doubling  constant of $X$. 
	Let $ 1< p< t < \min\{ r, \log_2 ( \beta_\omega ^{-1} )  r  /n \}  <\infty$ or let $ 1 <p \le t <r =\infty $.
	Let $1 < q <\infty$.
	Then
	\begin{align*}
		\left( \dot{F \H}^{ - s,q' ,\psi,L}_{p ',t' ,r ',\omega ' } (X) \right) ^\ast =	\dot{\mathcal F}^{s,q,\psi,L}_{p,t,r,\omega} (X) .
	\end{align*}	
\end{theorem}

\begin{proof}
	First we prove the inclusion $\dot{\mathcal F}^{s,q,\psi,L}_{p,t,r,\omega} (X)   \subset 	\left( \dot{F \H}^{ - s,q' ,\psi,L}_{p ',t' ,r ',\omega ' } (X) \right) ^\ast $. Let $ f \in \dot{\mathcal F}^{s,q,\psi,L}_{p,t,r,\omega} (X) $. 
	By (\ref{eq f = sum psi 2}), 	for every $\phi \in \mathcal S_\infty $, we have
	\begin{equation*}
		\langle f, \phi\rangle = \sum_{j= \mathbb Z } 	\langle \psi_j (\sqrt{L})  f, \psi_j (\sqrt{L}) \phi\rangle .
	\end{equation*}
	Then by Theorem \ref{predual scalar},	
	\begin{align*}
		|	\langle f, \phi\rangle | & \le \sum_{j= \mathbb Z } \int_X | \psi_j (\sqrt{L}) f (y) | | \psi_j (\sqrt{L}) \phi (y) | \d \mu (y) \\
		& \le \int_X   \left( \sum_{j= \mathbb Z } 2^{jsq} | \psi_j (\sqrt{L}) f (y) |^q\right)^{1/q}  \left(  \sum_{j= \mathbb Z } 2^{-jsq'}| \psi_j (\sqrt{L}) \phi (y) | ^{q'} \right) ^{1/q'}  \d \mu (y) \\
		& \le \| f \|_{\dot{\mathcal F}^{s,q,\psi,L}_{p,t,r,\omega} (X)   }  \left\|  \left(  \sum_{j= \mathbb Z } 2^{-jsq'}| \psi_j (\sqrt{L}) \phi (y) | ^{q'} \right) ^{1/q'}  \right\|_{ \H _{p', \omega '  }^{t',  r' } } \\
		& =  \| f \|_{\dot{\mathcal F}^{s,q,\psi,L}_{p,t,r,\omega} (X)   }  \left\| \phi   \right\|_{ \dot{F \H}^{ - s,q' ,\psi,L}_{p ',t' ,r ',\omega ' } (X) }.
	\end{align*}
	Hence we obtain 
	\begin{equation*}
		\| f \|_{ \left( \dot{F \H}^{ - s,q' ,\psi,L}_{p ',t' ,r ',\omega ' } (X) \right) ^\ast }  \le \| f\|_{ \dot{\mathcal F}^{s,q,\psi,L}_{p,t,r,\omega} (X) } .
	\end{equation*}

	Next we proceed with the embedding $\left(\dot{F \H}^{ - s,q' ,\psi,L}_{p ',t' ,r ',\omega ' } (X)  \right) ^\ast  \subset \dot{\mathcal F}^{s,q,\psi,L}_{p,t,r,\omega} (X)  $. 	
	Let $Y$ be the subspace of the normed space $  \H _{p', \omega '  }^{t',  r' }  ( \ell^{q'}) $ defined by
	\begin{equation*}
		Y:= \left\{     \{  2^{ - js} \psi_j (\sqrt{L}) g  \} _{j \in \mathbb Z } : g \in  \dot{F \H}^{ - s,q' ,\psi,L}_{p ',t' ,r ',\omega ' } (X)   \right\} .
	\end{equation*}
By Theorem \ref{block space is Banach}, $  \H _{p', \omega '  }^{t',  r' }  ( \ell^{q'}) $ is a Banach space.
	Note that the mapping $\rho:\dot{F \H}^{ - s,q' ,\psi,L}_{p ',t' ,r ',\omega ' } (X)   \to Y,$  $g  \mapsto \{  2^{ - js} \psi_j (\sqrt{L}) g  \} _{j \in \mathbb Z } $ is an isometric isomorphism.	
	If $f \in\left( \dot{F \H}^{ - s,q' ,\psi,L}_{p ',t' ,r ',\omega ' } (X)   \right) ^\ast $, we define $f_0 :=  f \circ  \rho^{-1} :  Y \to \mathbb C $ and note that $f_0 \in Y ^\ast  $ with $ \|f_0 \|_{ Y^\ast}  = \| f\|_{ \left( \dot{F \H}^{ - s,q' ,\psi,L}_{p ',t' ,r ',\omega ' } (X) \right) ^\ast }$. Then according to the Hahn-Banach Theorem there exists an
	extension $\tilde  f  \in   \left(   \H _{p', \omega '  }^{t',  r' } (\ell^{q'} ) \right) ^\ast $  of $f_0$ preserving the operator norm i.e., 
	\begin{equation*}
		\| \tilde f \|_{  \left(   \H _{p', \omega '  }^{t',  r' } (\ell^{q'} ) \right) ^\ast  } = \|f_0 \|_{Y^\ast }  = \| f\|_{  \left( \dot{F \H}^{ - s,q' ,\psi,L}_{p ',t' ,r ',\omega ' } (X)  \right) ^\ast  }.
	\end{equation*}
	Moreover, by Theorem \ref{predual vector}, there exists $  \{  f_j \}_{j\in \mathbb Z} \in  M_{p,\omega}^{t,r} (\ell^q ) $ representing $\tilde f$, that is 
	\begin{equation*}
		J_{\tilde f} (\vec  g) =   \int_X \sum_{j  \in \mathbb Z} f_j (x) g_j (x) \d \mu (x)
	\end{equation*}
	for each $\vec g  = \{  g_j \}_{j\in \mathbb Z} \in  \H _{p', \omega '  }^{t',  r' } (\ell^{q'} )  $ and
	\begin{equation} \label{H ell  ast = M ell}
		 \| \tilde f \| _{ \left(  \H _{p', \omega '  }^{t',  r' } (  \ell^{q'} ) \right) ^\ast  } = \|  \{  f_j \}_{j\in \mathbb Z} \|_{  M_{p,\omega}^{t,r} (\ell^q )  }  .
	\end{equation}	
Similarly to (\ref{f = sum_j in S _infty dual}) we can show
	\begin{equation} \label{f = sum_j in S _infty dual 2}
		f = \sum_{j \in \mathbb Z}  \psi_j (\sqrt L )  (2^{-js} f_j  )
	\end{equation}
	in the sense of $\mathcal S_\infty' $.
It remains to be proven that $f \in \dot{\mathcal F}^{s,q,\psi,L}_{p,t,r,\omega} (X)$  and 
	\begin{equation*}
		\| f \|_{\dot{\mathcal F}^{s,q,\psi,L}_{p,t,r,\omega} (X) } \lesssim \| f\|_{  \left( \dot{F \H}^{ - s,q' ,\psi,L}_{p ',t' ,r ',\omega ' } (X)   \right) ^\ast}.
	\end{equation*}
	For each $\ell \in \mathbb Z$, we have 
	\begin{align*}
		\psi_\ell (\sqrt{ L}) f (x)& =  \sum_{  j = \ell -1 } ^{\ell  +1} 	\psi_\ell (\sqrt{ L})  \psi_j  (\sqrt{ L}) (2^{-js}  f_j  ) (x) .
	\end{align*}
Just as in (\ref{psi ell psi j control by M}), for all $|\ell -j| \le 1$,
	\begin{equation*} 
	| 	\psi_\ell (\sqrt{ L})  \psi_j  (\sqrt{ L}) (  2^{-js} f_j)  (x)  | \lesssim \M (2^{-js} f_j  ) (x).
\end{equation*}
Thus by Lemma \ref{M bour weight} and Equation (\ref{H ell  ast = M ell}),
\begin{align*}
		\| f \|_{\dot{\mathcal F}^{s,q,\psi,L}_{p,t,r,\omega} (X) } & = \left\|  \left(   \sum_{\ell \in \mathbb Z} 2^{s\ell q} |\psi_\ell (\sqrt{L}) (f) |^q \right)^{1/q}     \right\|_{ M_{p,\omega}^{t,r} } \\
		&\lesssim \left\|  \left(   \sum_{\ell \in \mathbb Z} 2^{s\ell q} \left( \sum_{  j = \ell -1 } ^{\ell  +1} 	| \psi_\ell (\sqrt{ L})  \psi_j  (\sqrt{ L}) (2^{-js}  f_j  ) |  \right) ^q \right)^{1/q}     \right\|_{ M_{p,\omega}^{t,r} } \\
		& \lesssim  \left\|  \left(   \sum_{\ell \in \mathbb Z} 2^{s\ell q} \left( \sum_{  j = \ell -1 } ^{\ell  +1} 	\M (2^{-js} f_j  )   \right) ^q \right)^{1/q}     \right\|_{ M_{p,\omega}^{t,r} } \\
			& \lesssim  \left\|  \left(   \sum_{j \in \mathbb Z}  |f_j |^q   \right)^{1/q}     \right\|_{ M_{p,\omega}^{t,r} } = \| \tilde f \| _{ \left(  \H _{p', \omega '  }^{t',  r' } (  \ell^{q'} ) \right) ^\ast  }. 
\end{align*}
Thus we finish the proof.
\end{proof}

Finally we show that the preduals of weighted homogeneous Bourgain-Morrey Besov-Triebel-Lizorkin spaces are independent of the choice of   partitions of unity.
\begin{lemma}[Theorem 3.6, \cite{BX25}] \label{independet BTL}
	Let  $\psi$ and $\varphi$ be partitions of unity.
	Let $0< q\le \infty $, $s\in  \mathbb R$. Let $ 0<p <\infty $. Let $ \omega \in A_\infty  $ such that  $\omega \in A_{p/A} $ for some $A>0$. 	Let $\beta = 1 - ( 1-\alpha_0 )^{p/A}/ [\omega]_{A_ {p/A}} $.
	Let $0<p <t <\infty $  and $ \max\{ t, -nt/ \log_2 \beta \} <r <\infty   $, or let $ 0<p \le t <r =\infty $.
Then, the spaces $\dot{\mathcal A}^{s,q,\psi,L}_{p,t,r,\omega} (X) $   and  $\dot{\mathcal A}^{s,q,\varphi,L}_{p,t,r,\omega} (X) $, $A \in \{ B,F\}$  coincide with equivalent norms.
\end{lemma}
By Lemma \ref{independet BTL} and Theorems \ref{predual Besov}, \ref{predual TL}, we obtain the following result.
\begin{theorem}
		Let  $\psi$ and $\varphi$ be partitions of unity.		
			Let   $s\in \mathbb R$ and $\omega \in A_p$ for $p \in (1,\infty)$.	
		Let $\beta_\omega : =  1 - (1- \alpha_0 ) ^{ p }  / [\omega ]_{A_  {p} }$ where $\alpha_0$  is the reverse doubling  constant of $X$. 
		Let $ 1< p< t < \min\{ r, \log_2 ( \beta_\omega ^{-1} )  r  /n \}  <\infty$ or let $ 1 <p \le t <r =\infty $.
		
		{\rm (i)} 	
		Let $q' \in (0,\infty)$.	
		Then	
		 $\dot{B \H}^{ - s,q' ,\psi,L}_{p ',t' ,r ',\omega ' } (X) $ and  $	\dot{B \H}^{ - s,q' ,\varphi,L}_{p ',t' ,r ',\omega ' } (X) $   coincide with equivalent norms.

	{ \rm  (ii)}
	Let $q' \in (1,\infty)$.
Then 
 $\dot{F \H}^{ - s,q' ,\psi,L}_{p ',t' ,r ',\omega ' } (X) $ and  $	\dot{F \H}^{ - s,q' ,\varphi,L}_{p ',t' ,r ',\omega ' } (X) $   coincide with equivalent norms.
\end{theorem}

\noindent\textbf{Author Contributions}\quad 
All authors developed and discussed the results and contributed to the final
manuscript.

\medskip

\noindent\textbf{Data Availability}\quad Data sharing is not applicable 
to this article as no data sets were generated or analyzed.
\medskip

\noindent\textbf{Acknowledgments}\quad The authors used the free version of ChatGPT to assist with checking the proofs and English grammar of this manuscript.

\section*{Declarations}

\noindent\textbf{Conflict of interest}\quad All authors state no conflict of interest.

\medskip

\noindent\textbf{Informed Consent}\quad Informed consent has been obtained 
from all individuals included in this research work.


\bigskip

\noindent   Tengfei Bai, Pengfei Guo

\medskip

\noindent College of Mathematics and Statistics, Hainan Normal University, Haikou, Hainan 571158,
China

\medskip

\noindent Jingshi Xu (Corresponding author)

\medskip

\noindent School of Mathematics and Computing Science, Guilin University of Electronic Technology, Guilin 541004, China

\noindent Center for Applied Mathematics of Guangxi (GUET), Guilin 541004, China

\noindent Guangxi Colleges and Universities Key Laboratory of Data Analysis and Computation, Guilin 541004, China

\smallskip

\noindent {\it E-mails}:
\texttt{202311070100007@hainnu.edu.cn, baitengfei1@126.com} (T. Bai)

\noindent\phantom{{\it E-mails:}}
\texttt{guopf999@163.com} (P. Guo)

\noindent\phantom{{\it E-mails:}}
\texttt{jingshixu@126.com} (J. Xu)

\end{document}